\pdfoutput=1
\RequirePackage{ifpdf}
\ifpdf 
\documentclass[pdftex]{sigma}
\else
\documentclass{sigma}
\fi

\numberwithin{equation}{section}

\newtheorem{Theorem}{Theorem}[section]
\newtheorem{Proposition}[Theorem]{Proposition}
 { \theoremstyle{definition}
\newtheorem{Definition}[Theorem]{Definition}
\newtheorem{Example}[Theorem]{Example}
\newtheorem{Remark}[Theorem]{Remark} }
\newcommand\fg{{\frak{g}}}
\newcommand\fh{{\frak h}}

\newcommand\fm{{\frak m}}
\newcommand\fn{{\frak n}}
\newcommand\fp{{\frak p}}
\newcommand\fa{{\frak a}}
\newcommand\fb{{\frak b}}
\newcommand\fd{{\frak d}}

\newcommand\fs{{\frak s}}

\newcommand\fu{{\frak u}}

\newcommand\ba{{\mathbf a}}
\newcommand\bA{{\mathbf A}}
\newcommand\bx{{\mathbf x}}
\newcommand\bu{{\mathbf u}}
\newcommand\by{{\mathbf y}}
\newcommand\bn{{\mathbf n}}
\newcommand\bv{{\mathbf v}}
\newcommand{\fgl}{\mathop{{\frak g \frak l}}}
\newcommand{\fsl}{\mathop{{\frak s \frak l}}}

\newcommand{\fco}{\mathop{{\frak c \frak o}}}
\newcommand{\fso}{\mathop{{\frak s \frak o}}}

\newcommand\RR{{\mathbb R}}

\newcommand{\GL}{\mathop{{\rm GL}}}
\newcommand{\SO}{\mathop{{\rm SO}}}

\newcommand{\Spin}{\operatorname{Spin}}
\newcommand{\G}{\mathop{{{\rm G}_{2}^*}}}
\newcommand\ip{{\langle\cdot \,,\cdot \rangle}}
\newcommand{\Span}{\operatorname{span}}
\newcommand{\tr}{\operatorname{tr}}
\newcommand{\diag}{\operatorname{diag}}

\begin{document}
\allowdisplaybreaks

\newcommand{\arXivNumber}{1705.00023}

\renewcommand{\PaperNumber}{081}

\FirstPageHeading

\ShortArticleName{Local Type I Metrics with Holonomy in ${\rm G}_{2}^*$}

\ArticleName{Local Type I Metrics with Holonomy in $\boldsymbol{{\rm G}_{2}^*}$}

\Author{Anna FINO~$^\dag$ and Ines KATH~$^\ddag$}

\AuthorNameForHeading{A.~Fino and I.~Kath}

\Address{$^\dag$~Dipartimento di Matematica G. Peano, Universit\`a di Torino,\\
\hphantom{$^\dag$}~Via Carlo Alberto 10, Torino, Italy}
\EmailD{\href{mailto:annamaria.fino@unito.it}{annamaria.fino@unito.it}}

\Address{$^\ddag$~Institut f\"ur Mathematik und Informatik, Universit\"at Greifswald,\\
\hphantom{$^\ddag$}~Walther-Rathenau-Str. 47, D-17487 Greifswald, Germany}
\EmailD{\href{mailto:ines.kath@uni-greifswald.de}{ines.kath@uni-greifswald.de}}

\ArticleDates{Received October 24, 2017, in final form July 29, 2018; Published online August 03, 2018}

\Abstract{By [arXiv:1604.00528], a list of possible holonomy algebras for pseudo-Rieman\-nian manifolds with an indecomposable torsion free ${\rm G}_{2}^*$-structure is known. Here indecomposability means that the standard representation of the algebra on ${\mathbb R}^{4,3}$ does not leave invariant any proper non-degenerate subspace. The dimension of the socle of this representation is called the type of the Lie algebra. It is equal to one, two or three. In the present paper, we use Cartan's theory of exterior differential systems to show that all Lie algebras of Type I from the list in [arXiv:1604.00528] can indeed be realised as the holonomy of a~local metric. All these Lie algebras are contained in the maximal parabolic subalgebra $\mathfrak p_1$ that stabilises one isotropic line of ${\mathbb R}^{4,3}$. In particular, we realise $\mathfrak p_1$ by a local metric.}

\Keywords{holonomy; pseudo-Riemannian manifold; exterior differential system; torsion-free G-structures}

\Classification{53C29; 53C50; 53C10}

\section{Introduction}

Many problems in geometry can be rephrased as the problem of locally prescribing a given group as holonomy, and this can be reduced to a PDE problem in a number of ways, but most of these lead to PDE that are either degenerate or overdetermined in some way, so the methods of exterior differential systems turn out to be essentially involved. In \cite{Bryant} Bryant showed that the local existence of many non-Riemannian holonomy groups is based on the translation of the structure equations for a given holonomy group into an exterior differential system, and then to use Cartan--K\"ahler theory~\cite{BCG} to conclude the existence of such metrics. We recall that an exterior differential system is a system of equations on a manifold defined by equating to zero a~number of exterior differential forms.

\looseness=1 In the present paper, we consider $7$-dimensional pseudo-Riemannian manifolds $(M, g)$ with holonomy $H$ contained in the non-compact subgroup $\G \subset \SO(4, 3)$. This is the pseudo-Riemannian analogue of a torsion-free ${\rm G}_2$-structure, which is well known from the holonomy theory of Riemannian manifolds since ${\rm G}_2$ is one of the groups on Berger's list~\cite{Berger}. While torsion-free ${\rm G}_2$-structures exist on Riemannian 7-manifolds, their pseudo-Riemannian analogues are structures on manifolds of signature $(4,3)$. The Lie group $\G$ can be viewed either as the stabiliser of a certain generic 3-form, the stabiliser of a non-isotropic element of the real spinor representation of $\Spin(4, 3)$ or the stabiliser of a cross product on ${\mathbb R}^{4,3}$. Therefore a torsion-free $\G$-structure on a pseudo-Riemannian manifold M of signature $(4, 3)$ can be understood as a~parallel generic $3$-form, a parallel non-isotropic spinor field or a parallel cross-product~$\times$ on~M.

Examples of signature (4,3)-metrics with holonomy group equal to $\G$ have been constructed, see for instance \cite{ALN, CLSS, FL, GW, LP, W2}. In the present paper we are interested in the case where the holonomy is strictly contained in~$\G$.

In \cite{FK} a classification of indecomposable Berger algebras strictly contained in the Lie algebra~$\mathfrak g_2^*$ was obtained, where by indecomposable we mean that the holonomy representation does not leave invariant any proper non-degenerate subspace. The indecomposable Berger algebras $\fh \subset \mathfrak g_2^* \subset \mathfrak{so}(4, 3)$ have been distinguished by the dimension of the socle, i.e., of the maximal semisimple subrepresentation of their natural representation on ${\mathbb R}^{4,3}$. As in \cite{FK} we will say that $\fh$ is of Type I, II or III, if, respectively, the dimension of the socle is one, two or three. The types can be described in terms of the two 9-dimensional maximal parabolic subalgebras $\frak p_1$, $\frak p_2 $ of $\frak g_2^*$, which can be characterised by the action of $\G$ on isotropic subspaces of $\mathbb R^{4,3}$. More precisely, $\frak p_1$ is the Lie algebra of the stabiliser of an isotropic line and $\frak p_2$ is the stabiliser of an isotropic 2-plane spanned by two vectors $b_1$, $b_2$ satisfying $b_1 \times b_2 = 0$. Moreover, $\frak p_1 = {\frak {gl}} (2,\mathbb R) \ltimes \frak m$, where $\frak m$ is three-step nilpotent, is of Type I and $\frak p_2 = {\frak {gl}} (2,\mathbb R) \ltimes \frak n$, where $\frak n$ is two-step nilpotent, is of Type II.
Both $\frak p_1$ and $\frak p_2$ are indecomposable Berger algebras.

By \cite{FK}, $\frak h$ is contained, up to conjugation, in $\frak p_1$ if $\frak h$ is of Type I or III. If $\frak h$ is of Type II, then $\frak h \subset \frak p_2$ up to conjugation. In this way we obtained the list of possible holonomy algebras for pseudo-Riemannian manifolds with a torsion free $\G$-structure. To complete the classification of holonomy algebras one must prove that all Berger algebras can be realised as holonomy algebras. In~\cite{FK}, we already did this for some of the Berger algebras, which we realised as holonomy algebras of left-invariant metrics on Lie groups. For Type I, we realised the nilpotent Lie algebra $\frak m$ and, furthermore, a 7-dimensional solvable Lie algebra and a 6-dimensional nilpotent Lie algebra. For Type~II we realised $\frak n$, ${\frak {sl}}(2, \mathbb R) \ltimes \frak n$ and 3-dimensional abelian example, and for Type III a~three-dimensional abelian Lie algebra. Moreover, we know which Berger algebras are holonomy algebras of symmetric spaces with an invariant $\G$-structure~\cite{FK, K}. Furthermore, Willse proved for some of the Berger algebras that they appear as the holonomy of an ambient metric. For instance, $\fn$ arises as a holonomy algebra in this way~\cite{W}.

In the present paper we construct local metrics that have holonomy algebras of Type I. Since the list of Berger algebras of Type I is rather long, we will give all the details only for those Berger algebras that are maximal or minimal in the following sense. We consider the largest Berger algebra of Type~I, i.e., the parabolic subalgebra $\frak p_1$ and we consider all Berger algebras that do not contain all of $\fm$. For each Lie algebra $\fh$ of this kind we use an adapted coframe to write the structure equations for the $H$-structure and Cartan's theory of exterior differential systems. This will give a normal form for a metric whose holonomy is contained in the considered Berger algebra $\fh$. Then we use this normal form to find a metric whose holonomy is equal to~$\fh$. For each of the remaining Berger algebras $\fh$ from the list we will give only an example of a~metric that shows that $\fh$ is indeed a holonomy algebra. In this way we prove
\begin{Theorem}\label{T}
Each indecomposable Berger algebra of Type I is the holonomy algebra of a~metric of signature $(4,3)$.
\end{Theorem}

\noindent
{\bf Notation.} If $b_1,\dots,b_n$ is a basis of a vector space $W$, then we denote by $b^1,\dots,b^n$ its dual basis of $W^*$. Furthermore, $b^{i_1\dots i_k}:=b^{i_1}\wedge\dots\wedge b^{i_k}\in\bigwedge^k W^*$.

\section[Holonomy representations contained in $\G$ and their type]{Holonomy representations contained in $\boldsymbol{\G}$ and their type}

The group $\G$ is known to be one of the groups on Berger's list of the holonomy groups of irreducible pseudo-Riemannian manifolds. In general, the holonomy representation of a pseudo-Riemannian metric is not completely reducible since it can leave invariant isotropic subspaces. Hence one is interested not only in irreducible holonomy representations but also in the much wider class of indecomposable ones, that is, in those whose natural representation on the tangent space does not contain any non-degenerate invariant subspace. Lie algebras of holonomy groups are called holonomy algebras and their natural representation on the tangent space is called holonomy representation. We are interested in this representations rather than in the holonomy algebra as an abstract Lie algebra. Therefore we consider the holonomy algebra of a pseudo-Riemannian manifold of signature $(p,q)$ always as a subalgebra of~$\fso(p,q)$. Analogously to holonomy groups, it is called indecomposable if it does not leave invariant any non-degenerate subspace of~$\RR^{p,q}$.

This paper is part of a project whose aim is the classification of indecomposable holonomy algebras strictly contained in the Lie algebra $\fg^*_2$ of $\G$. Here we want to understand $\G$ as the stabiliser of the 3-form
\begin{gather}\label{Eomega1}
\omega=\sqrt 2\big(e^{167}+e^{235}\big)-e^4\wedge \big(e^{15}-e^{26}-e^{37}\big)
\end{gather}
on $\RR^7$, where $e_1,\dots,e_7$ is a basis of $\RR^7$. This 3-form induces a metric
\begin{gather}
\ip = 2\big(e^1\cdot e^5+e^2\cdot e^6+e^3\cdot e^7\big)- \big(e^4\big)^2 \label{Eip1}
\end{gather}
of signature $(4,3)$ on $\RR^7$. In particular, we consider $\G$ as a subgroup of $\SO(4,3)$.

In \cite{FK}, we obtained a classification of indecomposable Berger algebras strictly contained in the Lie algebra $\fg^*_2\subset \fso(4,3)$ on $\RR^{4,3}$. We distinguished three types of such algebras corresponding to the dimension of the maximal semisimple subrepresentation of their natural representation on~${\mathbb R}^{4,3}$. This subrepresentation is called socle. If a holonomy algebra $\fh$ is indecomposable, then its socle is (totally) isotropic. Hence, for $\fh\subset\fg_2^*$ it has dimension one, two or three. Accordingly, we will say that $\frak h$ is of Type I, II or III.

Here we want to study the question whether Berger algebras of Type I are indeed holonomy algebras, i.e., whether there are local metrics such that these Berger algebras are the holonomy algebras of these metrics. Let us first recall the classification of indecomposable Berger algebras of Type~I.

Let $\fh$ be the holonomy algebra of a pseudo-Riemannian manifold $M^{4,3}$ of signature $(4,3)$. Here we always assume that $\fh$ is indecomposable. Suppose that the holonomy representation of~$\fh$ on $V:=T_{x_0}\big(M^{4,3}\big)$ is contained in $\fg_2^*\subset \fso(V,\ip)$, where $\ip$ denotes the metric of~$M^{4,3}$ at~$x_0$. If $\fh$ is of Type~I, then we may choose a basis of $V$ which gives us an identification $V\cong\RR^7$ such that~(\ref{Eomega1}) and~(\ref{Eip1}) hold and, in addition, such that~$\fh$ is a subalgebra of the maximal parabolic subalgebra
\begin{gather}\label{EhI}
\fh^I:=\big\{ h(A,v,u,y)\,|\, A\in\fgl(2,\RR),\, v\in\RR,\, u, y \in\RR^2\big\}\cong \fp_1
\end{gather}
of $\fg_2^*$, where
\begin{gather*}h(A,v, u,y)= \left(
\begin{matrix}
\tr A &-u_2&u_1&\sqrt 2 v&0&-y_1&-y_2\\
0&a_1&a_2&\sqrt 2 u_1&y_1&0&v\\
0&a_3&a_4&\sqrt 2 u_2&y_2&-v&0\\
0&0&0&0&\sqrt 2 v&\sqrt 2 u_1&\sqrt 2 u_2\\
0&0&0&0&-\tr A&0&0\\
0&0&0&0&u_2&-a_1&-a_3\\
0&0&0&0&-u_1&-a_2&-a_4
	\end{matrix}\right)\end{gather*}
for	$ A:=\left(\begin{smallmatrix} a_1&a_2\\a_3&a_4\end{smallmatrix}\right)$, $y=(y_1,y_2)^\top$, $u=(u_1,u_2)^\top$.

We define
\begin{gather*} \fm:=\big\{h(0,v,u,y)\,|\, v\in\RR,\ u, y \in\RR^2\big\}\subset \fh^I\end{gather*}
and identify $\fgl(2,\RR)$ with $\{h(A,0,0,0)\,|\, A\in \fgl(2,\RR)\}$. Then $\fh^I=\fgl(2,\RR)\ltimes \fm$, where $A\in\fgl(2,\RR)$ acts on $\fm$ by
\begin{gather*}
A\cdot h(0,v,u,y)= h\big(0,\tr(A)v,Au,(A+\tr A)y\big).
\end{gather*}
The Lie bracket on $\fm$ is given by
\begin{gather*}
[h(0,v,u,y),h(0,\bar v,\bar u,\bar y)]=h\big(0,2\theta(u,\bar u),0,3(\bar v u-v\bar u)\big),
\end{gather*}
where $\theta(u,\bar u):=u_1\bar u_2-u_2\bar u_1$ for $u,\bar u\in\RR^2$. We define subspaces of $\fm$ by
\begin{gather*}
\fm(1,0,0):= \{h(0,v,0,0)\,|\, v\in\RR\},\qquad
\fm(0,1,0):= \big\{h\big(0,0,(u_1,0)^\top, 0\big)\,|\, u_1\in\RR\big\},\\
\fm(0,0,2):= \big\{h(0,0,0, y)\,|\, y\in\RR^2\big\}.
\end{gather*}
Now we put
\begin{gather*}\fm(i,j,2)= \fm(i,0,0)\oplus\fm(0,j,0)\oplus \fm(0,0,2)\end{gather*}
for $i,j\in\{0,1\}$.

\begin{Remark}
It is well known that the Lie algebra $\fg_2^*$ has a grading $\fg_2^*=g_{-3}\oplus g_{-2}\oplus g_{-1}\oplus g_0 \oplus g_{1}\oplus g_2\oplus g_3$. For details see, e.g., \cite[p.~14]{HS}, where the notation is very similar to ours. The parabolic subalgebra $\fp_1$ is equal to $g_0 \oplus g_{1}\oplus g_2\oplus g_3$ and $\fm$ equals $g_{1}\oplus g_2\oplus g_3$. Moreover, $\fm(1,0,0)=g_2$, $\fm(0,0,2)=g_3$ and $\fm(0,1,0)$ is a subspace of $g_1$.
\end{Remark}

We define matrices
\begin{gather*}C_a:=\left(\begin{matrix} a&-1\\1&a\end{matrix}\right),\qquad S:=\left(\begin{matrix} 1&1\\0&1\end{matrix}\right),\qquad N:=\left(\begin{matrix} 0&1\\0&0\end{matrix}\right),
\end{gather*}
and the following Lie subalgebras of $\fgl(2,\RR)\subset \fh^I$:
\begin{gather*}
\fd := \{\diag(a,d)\,|\, a,d\in\RR\},\qquad
\fco(2) := \left\{ \left(\begin{matrix} a&-b\\b&a\end{matrix}\right) \Big| \, a,b\in\RR\right\},\\
\hat\fb_2 := \Span\{ I,N\},\qquad
\fs_\lambda := \Span \{\diag(\lambda, \lambda-1),\ N\},\qquad \lambda\in\RR,\\
\fb_2 := \text{Lie algebra of upper triangular matrices}.
\end{gather*}

Let $\fa$ be the projection of $\fh$ to $\fgl(2,\RR)\subset \fh^I$.

\begin{Theorem}[{{\cite{FK}}}]\label{T1}
If $\fh\subset \fg_2^*$ is an indecomposable Berger algebra of Type I, then there exists a basis of $V$ such that we are in one of the following cases
\begin{enumerate}\itemsep=0pt
\item[$1)$] $\fa\in \big\{0,\fsl(2,\RR),\, \fgl(2,\RR),\, \fco(2),\,\fb_2,\, \hat \fb_2,\, \fd,\,\RR\cdot C_a,\, \RR\cdot S\big\}$ and $\fh=\fa\ltimes \fm$,
\item[$2)$] $\fa=\fs_\lambda=\Span \{X:=\diag(\lambda,\lambda-1),\, N\}$ and
\begin{enumerate}\itemsep=0pt
\item[$a)$] $\lambda\in\RR$ and $\fh=\fa\ltimes \fm$,
\item[$b)$] $\lambda=1$ and $\fh=\RR\cdot h\big(X,0,(0,1)^\top,0\big)\ltimes (\RR\cdot N\ltimes \fm(1,1,2))$,
\item[$c)$] $\lambda=2$ and $\fh=\Span\big\{X,\, h\big(N,0,(0,1)^\top,0\big) \big\}\ltimes\fm(i,j,2)$, where $(i,j)\in\{(0,0),(1,0)$, $(1,1)\}$,
\end{enumerate}
\item[$3)$] $\fa=\RR\cdot\diag(1,\mu)$ and
\begin{enumerate}\itemsep=0pt
\item[$a)$] $\mu\in [-1,1]$ and $\fh=\fa\ltimes \fm$,
\item[$b)$] $\mu=0$ and $\fh=\RR\cdot h\big(\diag(1,0),0,(0,1)^\top,0\big)\ltimes \fm(1,1,2)$,
\end{enumerate}
\item[$4)$] $\fa=\RR\cdot N$ and
\begin{enumerate}\itemsep=0pt
\item[$a)$] $\fh=\fa\ltimes \fm$,
\item[$b)$] $\fh=\RR\cdot h\big(N,0,(0,1)^\top,0\big)\ltimes \fm(1,j,2)$ for $j\in\{0,1\}$.
\end{enumerate}
\end{enumerate}
\end{Theorem}

Our aim is to show that all these Lie algebras are indeed holonomy algebras of a metric of signature $(4,3)$ (in particular, they are all Berger algebras). First we will concentrate on maximal and minimal examples in the following sense. We realise as a holonomy algebra the parabolic subalgebra $\fh^I\cong\frak p_1$ and all Berger algebras that do not contain all of $\fm$. In all these cases we give a kind of local normal form for a metric whose holonomy is contained in the considered Berger algebra $\fh$ and we give an example of a metric with holonomy equal to~$\fh$. For each of the remaining Berger algebras we will give only an example of a metric that shows that this Lie algebra is a holonomy algebra. In this way we prove Theorem~\ref{T}.

Before we start, we mention that for some of the listed Berger algebras already metrics are known. For instance, in \cite{FK}, we constructed left-invariant metrics on Lie groups realising $\fm$, $\RR\cdot N\ltimes \fm$ and $\fs_{1/2}\ltimes \fm$ as holonomy algebras. Furthermore, Willse constructed an ambient metric with holonomy $\fs_{1/2}\ltimes \fm$ (personal communication).

We recall that in general, given an $n$-dimensional pseudo-Riemannian manifold $(M, g)$ with holonomy $H$, the holonomy bundle of the metric $g$ is always a $1$-flat $H$-structure, i.e., there exists a compatible torsion-free connection (see for instance \cite{Bryant}). The structure equations on the $H$-structure $B \rightarrow M$ are given by
\begin{gather*}
d \eta = - \theta \wedge \eta
\end{gather*}
and
\begin{gather*}
d \theta = - \theta \wedge \theta +R (\eta \wedge \eta),
\end{gather*}
where $\eta\colon TB \rightarrow {\mathbb R}^n$, $\theta\colon TB \rightarrow \frak h$ and $R\colon B \rightarrow K (\frak h)$ is the curvature function with $\frak h$ the Lie algebra of $H$. Here $K (\frak h)$ denotes the $H$-representation
\begin{gather*}
0 \rightarrow K (\frak h) \rightarrow S^2 (\frak h) \overset{\wedge}{\rightarrow} \Lambda^4 \big(\mathbb R^n\big).
\end{gather*}
Using a local frame $(b_i)$ trivializing the $1$-flat $H$-structure it is possible to write down locally the structure equations for the 1-flat $H$-structure and reformulate the structure equations as a~set of differential equations for local functions on the manifold. Note that in terms of the local frame we have that $\theta$ is given by $\theta^i_j=b^i(\nabla{b_j})$, where $\nabla$ is the Levi-Civita connection and the 1-form $b^i(\nabla{b_j})$ is defined by $b^i(\nabla{b_j})(X):=b^i(\nabla_X{b_j})$ for all tangent vectors $X$.

In order to show that all the Berger algebras $\frak h$ of type I can be realized as holonomy algebras, we first obtain a local normal form for metrics with holonomy algebra contained in the maximal parabolic subalgebra $\frak p_1$. Let $P_1$ be the Lie group with Lie algebra $\frak p_1$. Since the existence of a 1-flat $P_1$-structure is equivalent to the induced pseudo-Riemannian metric having holonomy algebra contained in $\frak p_1$, we first write down locally the structure equations for a 1-flat $P_1$-structure trivialized by an adapted local frame $(b_i)$ and reformulate these structure equations as a set of differential equations for local functions on the manifold. Then we give a solution to the mentioned differential equations. By using Ambrose--Singer holonomy theorem~\cite{AS}, we compute enough components of the curvature tensor and its covariant derivative to conclude that the holonomy algebra is equal to $\frak p_1$. We use the same argument for all indecomposable Berger algebras~$\frak h$ of type I which do not contain all of~$\frak m$. For the remaining Berger algebras~$\frak h$ we only give a~solution of the above mentioned differential equations which shows that~$\fh$ is a~holonomy algebra.

\section{Local metrics with holonomy of Type I} \label{S2}
\subsection{Adapted coordinates}
In this subsection, we introduce a normal form for metrics whose holonomy algebra is contained in~$\fh^I\cong\fp_1$. We will use this normal form in the next subsection to prove that $\fh^I$ is indeed a~holonomy algebra. Let $b_1,\dots,b_7$ be a local section in the reduction of the frame bundle of~$M^{4,3}$ to the holonomy group. With respect to $b_1,\dots,b_7$ the 3-form defining the $\G$-structure equals
\begin{gather*}
\omega=\sqrt 2\big(b^{167}+b^{235}\big)-b^4\wedge \big(b^{15}-b^{26}-b^{37}\big)
\end{gather*}
and the metric is given by
\begin{gather}
g=2\big(b^1\cdot b^5+b^2\cdot b^6+b^3\cdot b^7\big)- \big(b^4\big)^2. \label{Eip}
\end{gather}
Let $\fh^I$ be defined as in equation~(\ref{EhI}). The dual frame $b^1,\dots,b^7$ satisfies the structure equations
\begin{gather}\label{Ese}
\left(\begin{matrix} db^1\\db^2\\db^3\\db^4\\db^5\\db^6\\db^7\end{matrix}\right) = -\left(
\begin{matrix}
\tr \bA &-\bu_2&\bu_1&\sqrt 2 \bv&0&-\by_1&-\by_2\\
0&\ba_1&\ba_2&\sqrt 2 \bu_1&\by_1&0&\bv\\
0&\ba_3&\ba_4&\sqrt 2 \bu_2&\by_2&-\bv&0\\
0&0&0&0&\sqrt 2 \bv&\sqrt 2 \bu_1&\sqrt 2 \bu_2\\
0&0&0&0&-\tr \bA&0&0\\
0&0&0&0&\bu_2&-\ba_1&-\ba_3\\
0&0&0&0&-\bu_1&-\ba_2&-\ba_4
	\end{matrix}\right) \wedge \left(\begin{matrix} b^1\\b^2\\b^3\\b^4\\b^5\\b^6\\b^7\end{matrix}\right)
\end{gather}	
for	$ \bA:=\left(\begin{smallmatrix} \ba_1&\ba_2\\ \ba_3&\ba_4\end{smallmatrix}\right)$, where bold symbols denote 1-forms.

\begin{Definition}Let $x_1,\dots, x_7$ be (local) coordinates on $M^{4,3}$. We introduce the notation \begin{gather*}dx_{(i_1,\dots,i_k)}:=dx_{i_1}\wedge \dots \wedge dx_{i_k}.\end{gather*}
We denote by $I(\omega^1,\dots, \omega^k)$ the algebraic ideal generated by the differential forms $\omega^1,\dots, \omega^k$ and define \begin{gather*}I_0:=I(dx_5,dx_6,dx_7),\qquad I_1:=I\big(dx_{(5,6)}, dx_{(6,7)}, dx_{(5,7)}\big).\end{gather*}
\end{Definition}

\begin{Proposition}\label{Lg}The holonomy of $\big(M^{4,3},g\big)$ is contained in $\fh^I$ if and only if there are local coordinates $x_1,\dots,x_7$ such that $g=2\big(b^1\cdot b^5+b^2\cdot b^6+b^3\cdot b^7\big)- \big(b^4\big)^2$ for
\begin{gather}
b^1 = dx_1+r_5(x_1,\dots,x_7)dx_5+r_6(x_1,\dots,x_7)dx_6+r_7(x_1,\dots,x_7)dx_7,\label{beg}\\
b^2 =dx_2+ q_2(x_2,\dots,x_7)dx_6+s_2(x_2,\dots,x_7)dx_7,\\
b^3 =dx_3+ q_3(x_2,\dots,x_7)dx_6+s_3(x_2,\dots,x_7)dx_7,\\
b^4 =dx_4+ q_4(x_5,x_6,x_7)dx_6,\\
b^5 =f(x_5,x_6,x_7)dx_5,\\
b^j =dx_j,\qquad j=6,7, \label{end}
\end{gather}
 where the functions $r_5$, $r_6$, $r_7$, $q_2$, $q_3$, $q_4$, $s_2$, $s_3$ and $f$ satisfy the differential equations
\begin{gather}
(s_2-q_3)_{x_i}=0,\qquad i\in\{2,3\},\qquad (s_2-q_3)_{x_4}=-(q_4)_{x_7},\label{Eb1}\\
(r_6)_{x_1}=f_{x_6}/f=(q_2)_{x_2}+(q_3)_{x_3}, \qquad (r_7)_{x_1}=f_{x_7}/f=(q_3)_{x_2}+(s_3)_{x_3},\label{Eb2}\\
(r_6)_{x_2}=-\sqrt{2} (q_3)_{x_4},\qquad \, (r_7)_{x_2}=-\sqrt{2} (s_3)_{x_4},\label{Eb3}\\
(r_6)_{x_3}=\sqrt{2} (q_2)_{x_4},\qquad (r_7)_{x_3}=\sqrt{2} ((q_3)_{x_4}-(q_4)_{x_7}),\label{Eb4}\\
(r_6)_{x_4}= -2\sqrt{2}\left((s_2)_{x_6}-(q_2)_{x_7}-\sum_{j=2}^4(s_2)_{x_j}q_j+\sum_{j=2}^3(q_2)_{x_j}s_j\right)-(q_4)_{x_5}/f ,\label{Eb4a}\\
(r_7)_{x_4}= -2\sqrt{2}\left((s_3)_{x_6}-(q_3)_{x_7}-\sum_{j=2}^4(s_3)_{x_j}q_j+\sum_{j=2}^3(q_3)_{x_j}s_j\right) ,\label{Eb5}\\
(r_5)_{x_1}=- \tfrac f{\sqrt 2} (q_4)_{x_7}, \label{Eb6}\\
(r_5)_{x_2}=- \tfrac f{2\sqrt 2} (r_7)_{x_4}, \label{Eb7}\\
(r_5)_{x_3}=- \tfrac1{2\sqrt 2} (q_4)_{x_5}+ \tfrac f{2\sqrt 2} (r_6)_{x_4}, \label{Eb7a}\\
(r_5)_{x_4}= \tfrac f{\sqrt 2} \left( (r_6)_{x_7}-(r_7)_{x_6}+(r_7)_{x_1}r_6-(r_6)_{x_1}r_7+\sum_{j=2}^4 (r_7)_{x_j}q_j-\sum_{j=2}^3 (r_6)_{x_j}s_j\right) \nonumber\\
\hphantom{(r_5)_{x_4}=}{} + \tfrac 1{\sqrt 2} ((q_3)_{x_5}-(s_2)_{x_5}).\label{Eb10}
\end{gather}
\end{Proposition}
\begin{proof}First assume that the holonomy is contained in $\fh^I$. Let $b_1,\dots,b_7$ be a local section in the reduction of the frame bundle to the holonomy group. The structure equations imply that we can introduce coordinates such that $b^5=fdx_5$ and $b^6,b^7\in I_0$. We transform the local frame pointwise (in a smooth way) by a suitable element of \begin{gather*}\GL(2,\RR)\cong \big\{\diag\big(\det a, a,1, (\det a)^{-1}, \big(a^\top\big)^{-1}\big)\,|\, a\in\GL(2,\RR)\big\}\end{gather*} and, furthermore, by $\exp h(0,0,u,0)$ for a suitable local map $u\colon M^{4,3} \rightarrow \RR^2$ such that
\begin{gather*}b^6=dx_6,\qquad b^7=dx_7\end{gather*}
(possibly changing $f$). Then
\begin{gather*}
0 =ddx_6= -\bu_2\wedge fdx_5+\ba_1\wedge dx_6 +\ba_3\wedge dx_7,\\
0 =ddx_7= \bu_1\wedge fdx_5+\ba_2\wedge dx_6 +\ba_4\wedge dx_7
\end{gather*}
implies $\ba_1,\dots,\ba_4,\bu_1, \bu_2\in I_0$. Now
\begin{gather*}db^5=df\wedge dx_5=(\tr \bA)\wedge b^5= (\ba_1+\ba_4)\wedge fdx_5\end{gather*}
implies $f=f(x_5,x_6,x_7)$.
Furthermore, $db^4\in I(dx_5,dx_6)$. Hence $b_4=dx_4+\hat qdx_5+q_4 dx_6$. Transforming the local frame by $\exp h(0,v,0,0)$ for a suitable local function $v$ we obtain $\hat q=0$. Thus
\begin{gather*}b^4=dx_4+q_4 dx_6.\end{gather*}
Since
\begin{gather*} db^4=dq_4\wedge dx_6 =-\sqrt 2\big( \bv\wedge b^5+ \bu_1\wedge b^6 +\bu_2\wedge b^7\big)\in I(\bv\wedge dx_5)+ I_1,\end{gather*}
we have $q_4=q_4(x_5,x_6,x_7)$ and $\bv\in I_0$. Now we see from $\tr A, \bu_1, \bu_2, \bv, b^6, b^7 \in I_0$ that
\begin{gather*}db^1=-(\tr \bA)\wedge b^1 +\bu_2\wedge b^2-\bu_1\wedge b^3-\sqrt 2 \bv\wedge b^4+\by_1\wedge b^6 +\by_2\wedge b^7\end{gather*}
is in $I_0$. Thus $b^1=dx_1+ r_5 dx_5+ r_6 dx_6+r_7dx_7$. Similarly, we proceed with $b^2$ and $b^3$ and obtain
\begin{gather*}b^2=dx_2+\hat q_2dx_5+q_2 dx_6+s_2dx_7, \qquad b^3=dx_3+\hat q_3 dx_5 +q_3 dx_6+s_3dx_7.\end{gather*}
Transforming the local frame by $\exp h(0,0,0,y)$ for suitable local function $y\colon M^{4,3}\to \RR^2$ we may assume that $\hat q_{2}=\hat q_{3}=0$. By (\ref{Ese}), the functions $q_2$, $q_3$, $s_2$, $s_3$ do not depend on~$x_1$. This shows that under the assumption that the holonomy is in $\fh^I$, we find coordinates $x_1,\dots,x_7$ such that equations \eqref{Eip} and \eqref{beg}--\eqref{end} hold.

Now assume that we have any local metric $g$ given by \eqref{Eip} with respect to a local frame $b_1,\dots,b_7$. Denote the holonomy group of $g$ by $H$ and assume that $b_1,\dots,b_7$ is a section in the reduction of the frame bundle to $H$. Assume further that \eqref{beg}--\eqref{end} hold. Since the connection reduces to the holonomy reduction, the matrix-valued 1-form $\theta=(\theta^i_j)=(b^i(\nabla b_j))$ takes values in the Lie algebra $\fh$ of $H$. Hence $\fh $ is contained in~$\fh^I$ if and only if $\theta$ takes values in $\fh^I$. Thus, by~\eqref{EhI}, the holonomy of~$g$ is contained in $\fh^I$ if and only if
\begin{gather}
 b^i(\nabla b_{1})=0,\qquad i=2,3, \label{Es1}\\
 b^1(\nabla b_1)=b^2(\nabla b_2)+b^3(\nabla b_3), \label{Es2}\\
 b^2(\nabla b_4)=\sqrt 2 b^1(\nabla b_3),\qquad b^3(\nabla b_4)=-\sqrt 2 b^1(\nabla b_2), \label{Es3}\\
 b^1(\nabla b_4)=\sqrt 2b^2(\nabla b_7). \label{Es4}
\end{gather}
Note that there are six more equations which have to be satisfied, but these are automatically fulfilled with the ansatz of the adapted coframe before. Using the Koszul formula, we find
\begin{gather*}
b^i(\nabla b_{1}) = \tfrac12 ((r_{4+i})_{x_1}- f_{x_{4+i}}/{f}) b^5 ,\qquad i=2,3,\label{En1}\\
b^1(\nabla b_1) = \tfrac{(r_5)_{x_1}}{f} b^5 +\left(\tfrac{f_{x_6}}{2f}+\tfrac{(r_6)_{x_1}}{2}\right) b^6+\left(\tfrac{f_{x_7}}{2f}+\tfrac{(r_7)_{x_1}}{2}\right) b^7, \label{En2}\\
b^{1}(\nabla b_{i}) = \tfrac{(r_5)_{x_i}}f b^5 + \tfrac{ (r_6)_{x_i}}2 b^6 + \tfrac{(r_7)_{x_i}}2 b^7,\qquad i=2,3,\label{En2a}\\
b^{2}(\nabla b_{2}) = \tfrac12 (r_6)_{x_2} b^5+(q_2)_{x_2} b^6+\tfrac12\big((s_2)_{x_2}+(q_3)_{x_2}\big) b^7, \label{En3}\\
b^3(\nabla b_3) = \tfrac12 (r_7)_{x_3} b^5+\tfrac12\big((s_2)_{x_3}+(q_3)_{x_3}\big) b^6+(s_3)_{x_3} b^7 , \label{En4}\\
b^1(\nabla b_4) = \tfrac{(r_5)_{x_4}}{f} b^5 +\left(\tfrac{(q_4)_{x_5}}{2f}+\tfrac{(r_6)_{x_4}}{2}\right) b^6+ \tfrac{(r_7)_{x_4}}{2} b^7 , \label{En5} \\
b^2(\nabla b_4) = \left(-\tfrac{(q_4)_{x_5}}{2f}+\tfrac{(r_6)_{x_4}}{2}\right)b^5+ (q_2)_{x_4} b^6 + \tfrac12 \big({-}(q_4)_{x_7}+(s_2)_{x_4}+(q_3)_{x_4}\big) b^7 ,\label{En6}\\
b^3(\nabla b_4) = \tfrac12 (r_7)_{x_4} b^5+\tfrac12\big((q_4)_{x_7}+(s_2)_{x_4}+(q_3)_{x_4}\big) b^6+(s_3)_{x_4} b^7 , \label{En7}\\
b^2(\nabla b_7) = -\tfrac12 \sum_{i=2}^3(s_2-q_3)_{x_i} b^i -\tfrac12 \big((s_2-q_3)_{x_4}+(q_4)_{x_7}\big) b^4 \nonumber \\
\hphantom{b^2(\nabla b_7) =}{} +\tfrac12 \left( (r_6)_{x_7}-(r_7)_{x_6} +(r_7)_{x_1}r_6-(r_6)_{x_1}r_7+\sum_{j=2}^4 (r_7)_{x_j}q_j \right. \nonumber\\
 \left. \hphantom{b^2(\nabla b_7) =}{} -\sum_{j=2}^3 (r_6)_{x_j}s_j -(s_2)_{x_5}/f+(q_3)_{x_5}/f\right) b^5 \label{En8} \\
\hphantom{b^2(\nabla b_7) =}{} -\left((s_2)_{x_6}-(q_2)_{x_7}-\sum_{j=2}^4(s_2)_{x_j}q_j+\sum_{j=2}^3(q_2)_{x_j}s_j\right) b^6 \nonumber\\
\hphantom{b^2(\nabla b_7) =}{} -\left( (s_3)_{x_6}-(q_3)_{x_7}-\sum_{j=2}^4(s_3)_{x_j}q_j+\sum_{j=2}^3(q_3)_{x_j}s_j \right) b^7 .\nonumber
\end{gather*}
Hence, (\ref{Es4}) is equivalent to (\ref{Eb1}), (\ref{Eb4a}), (\ref{Eb5}) and (\ref{Eb10}). Furthermore, (\ref{Es3}) is equivalent to (\ref{Eb7}), (\ref{Eb7a}) and
\begin{gather}
 -\sqrt{2} (r_6)_{x_2}= (q_4)_{x_7}+(s_2)_{x_4}+(q_3)_{x_4},\qquad (r_7)_{x_2}=-\sqrt{2} (s_3)_{x_4},\label{Eb3p}\\
 (r_6)_{x_3}=\sqrt{2} (q_2)_{x_4},\qquad \sqrt{2}(r_7)_{x_3}= -(q_4)_{x_7}+(s_2)_{x_4}+(q_3)_{x_4}.\label{Eb4p}
\end{gather}
Moreover, (\ref{Es2}) is equivalent to
\begin{gather}
(r_5)_{x_1}/f = \tfrac12 \big( (r_6)_{x_2}+(r_7)_{x_3}\big),\label{Ebeg}\\
f_{x_6}/f+(r_6)_{x_1}=2(q_2)_{x_2}+(s_2)_{x_3}+(q_3)_{x_3},\\
f_{x_7}/f+(r_7)_{x_1}=2(s_3)_{x_3}+(s_2)_{x_2}+(q_3)_{x_2},\label{Eend}
\end{gather}
and (\ref{Es1}) to
\begin{gather}\label{Eb1p}
(r_6)_{x_1}=f_{x_6}/f, \qquad (r_7)_{x_1}=f_{x_7}/f.
\end{gather}
Using equation~(\ref{Eb1}), we see that (\ref{Eb3p})--(\ref{Eb1p}) is equivalent to (\ref{Eb2}), (\ref{Eb3}), (\ref{Eb4}) \linebreak and~(\ref{Eb6}).
 \end{proof}

 \subsection{Realisation of the maximal parabolic subgroup as a holonomy group}\label{S3.2}
 It is easy to check that the functions
 \begin{gather*}
 u(x_5,x_6,x_7)=x_7,\qquad f=e^u,\qquad q_2(x_2,\dots,x_7) = -\sqrt 2x_4,\\
 q_3(x_2,\dots,x_7)= \sqrt 2\big(\tfrac12 x_5-x_6\big)+ \big( {-}x_4 +x_5+\tfrac 1{\sqrt 2}x_7 \big)x_6e^{-x_7},\\
 q_4(x_5,x_6,x_7)=x_5-2x_6+x_6e^{-x_7},\\
 s_2(x_2,\dots,x_7) =0,\qquad
 s_3(x_2,\dots,x_7) =-\tfrac1{\sqrt 2}x_4 +x_3 ,\\
 r_5(x_1,\dots,x_7) =\tfrac1{\sqrt 2} x_4(x_6-1)e^{-x_7} +x_7 + \tfrac1{\sqrt 2}(x_1 x_6 -x_3) -x_4^2 e^{x_7} ,\\
 r_6(x_1,\dots,x_7) =\big( {-}x_4+x_5+\sqrt2 x_6 -\tfrac1{\sqrt 2} x_7+\sqrt2 x_2 x_6 \big) e^{-x_7} + e^{-2x_7} -2x_3 ,\\
 r_7(x_1,\dots,x_7) = x_1+x_2
 \end{gather*}
 satisfy (\ref{Eb1})--(\ref{Eb10}). Hence the metric defined by (\ref{Eip}), where $b_1,\dots, b_7$ are given as in Proposition~\ref{Lg}, has holonomy $\fh\subset \fh_I$.

We want to show that $\fh$ is equal to $\fh^I$. More exactly, we will show that $R_{56}$, $R_{57}$ and $(\nabla _{b_5}R)_{56}$ generate $\fh^I$ as a Lie algebra. Since by the Ambrose--Singer theorem these operators are in $\fh\subset \fh^I$, this will prove the assertion. We will see that it is not necessary to know all components of these operators. For entries that are not needed, we just write a `$*$' (even if these entries are not complicated). We compute
\begin{gather*}
R_{56}= h\big(A_{56},*,\big( 0,-\tfrac 12\big)^\top\!,* \big), \qquad\!\!
R_{57}= h\big(A_{57},*,\big( \tfrac12,0)^\top\!,* \big) , \qquad\!\!
(\nabla _{b_5}R)_{56}= h(A'_{56},*,*,* ),
\end{gather*}
at $x=0$, where $ A_{56}=\left( \begin{smallmatrix} -\frac1{\sqrt 2}&0\\0&0\end{smallmatrix}\right)$, $ A_{57}=\left( \begin{smallmatrix}0&0\\-1&0\end{smallmatrix}\right)$, $ A'_{56}=\left( \begin{smallmatrix}0& -\frac1{\sqrt 2}\\1-\frac1{2\sqrt 2}&0\end{smallmatrix}\right)$. In particular, we get $\fa=\fgl(2,\RR)$. Since we do not know apriori that~$\fh$ is indecomposable, this is not sufficient to prove $\fh=\fh^I$. To finish the proof, we will show that \begin{gather*}\fu:= \big\{u\in\RR^2 \,|\, \exists\, v\in\RR,\,\exists\, y\in\RR^2\colon\, h(0,v,u,y)\in \fh\big\}\end{gather*} is equal to $\RR^2$, which implies that $\fh$ contains~$\fm$. Since $R_{56}$ and $R_{57}$ are in $\fh$, also
\begin{gather*}[R_{56},R_{57}]= \tfrac 1{\sqrt{2} }h\big(A_{57},*,\big({-} \tfrac12,0\big)^\top,*\big) \end{gather*}
is in $\fh$, thus $(1,0)^\top$ is in $\fu$. Hence $\fu=\RR^2$ because of the $\fa$-invariance of $\fu$.

\subsection{Type I 2(b)}
We consider the Lie algebra $\fh$ spanned by $h\big(\diag(1,0),0,(0,1)^\top,0\big)$, $N$ and $\fm(1,1,2)$.
The structure equations are now
\begin{gather}\label{EI2(b)}\left(\begin{matrix} db^1\\db^2\\db^3\\db^4\\db^5\\db^6\\db^7\end{matrix}\right) = -\left(
\begin{matrix}
{\bx} &-\bx&\bu_1&\sqrt2 \bv&0&-\by_1&-\by_2\\
0&\bx&\bn&\sqrt2 \bu_1&\by_1&0&\bv\\
0&0&0&\sqrt 2 \bx&\by_2&-\bv&0\\
0&0&0&0&\sqrt2\bv&\sqrt2 \bu_1&\sqrt 2 \bx\\
0&0&0&0&-\bx&0&0\\
0&0&0&0&\bx&-\bx&0\\
0&0&0&0&-\bu_1&-\bn&0
\end{matrix}\right) \wedge \left(\begin{matrix} b^1\\b^2\\b^3\\b^4\\b^5\\b^6\\b^7\end{matrix}\right).
\end{gather}
\begin{Proposition} \label{prop2.2} The holonomy of $\big(M^{4,3},g\big)$ is contained in the Lie algebra $\fh$ of Type~I~$2(b)$ if and only if there are local coordinates $x_1,\dots,x_7$ such that $g=2\big(b^1\cdot b^5+b^2\cdot b^6+b^3\cdot b^7\big)- \big(b^4\big)^2$ for
\begin{gather}
b^1 = dx_1+r_5(x_1,\dots,x_7)dx_5+r_6(x_2,\dots,x_7)dx_6+r_7(x_4,\dots,x_7)dx_7, \nonumber \\
b^2 =dx_2+ q_2(x_3,\dots,x_7)dx_6, \qquad b^3 =dx_3+ q_3(x_5,x_6,x_7)dx_6, \label{Eb}\\
b^6 =dx_6+p(x_5,x_6)dx_5, \qquad b^j =dx_j, \qquad j=4,5,7, \nonumber
\end{gather}
where the functions $q_3$, $r_5$, $r_6$, $r_7$ are of the form
\begin{gather*}
q_3= -p_{x_6}x_7+\bar q_3(x_5,x_6),\qquad
r_5= -p_{x_6} x_1 +p_{x_6}(1-p)x_2+\bar r_5(x_3,\dots,x_7),\\
r_6= -p_{x_6}x_2 +\bar r_6(x_3,\dots, x_7),\qquad
r_7= -2\sqrt2 p_{x_6}x_4 +\bar r_7(x_5,x_6,x_7).
\end{gather*}
and $p$, $q_2$, $q_3$, $\bar r_5$, $\bar r_6$, $\bar r_7$ satisfy the differential equations
\begin{gather}
(\bar r_6)_{x_3}= (q_2)_{x_3}p+ {\sqrt2}(q_2)_{x_4}, \label{Esysta}\\
(\bar r_6)_{x_4}=(q_2)_{x_4}p+2\sqrt 2(q_2)_{x_7},\label{Esyst2}\\
(\bar r_5)_{x_3} =\big((q_2)_{x_3}p+ {\sqrt2}(q_2)_{x_4}\big)p+ (q_2)_{x_7}, \label{Esyst3}\\
(\bar r_5)_{x_4}= \tfrac 1{\sqrt2} \big((\bar r_6)_{x_7} - (\bar r_7)_{x_6} +3 (q_2)_{x_7}p +(q_3)_{x_5}\big) +2p_{x_6 x_6} x_4 +(q_2)_{x_4} p^2.\label{Esyste}
\end{gather}
\end{Proposition}
\begin{proof} The structure equations imply $b^5=fdx_5$ for some function $f$. Transforming the local frame by $\exp h\big(\diag(x,0),0,(0,x)^\top,0\big)$ for a suitable local function $x$ we may assume \smash{$b^5=dx_5$}. Hence $\bx\in I(dx_5)$, which implies $db^6\in I(dx_5)$. Thus we can introduce $x_6$ such that $b^6=dx_6 +p dx_5$. Hence $p=p(x_5,x_6)$. Furthermore, (\ref{EI2(b)}) shows $db^7\in I(dx_5, dx_6)$. Hence we can introduce~$x_7$ such that $b^7=dx_7+f_5 dx_5+f_6 dx_6$. Transforming the local frame by $\exp h\big(n\cdot N,0,(u_1,0)^\top,0\big)$ for suitable local functions $u_1$, $n$ we may assume $b^7=dx_7$. By~(\ref{EI2(b)}), we obtain
\begin{gather*}0=db^7=\bu_1\wedge dx_5 +\bn \wedge (dx_6+p dx_5),\end{gather*}
which implies $\bn, \bu_1\in I(dx_5,dx_6)$. Consequently, $db^4\in I(dx_5)$. Hence we can introduce $x_4$ such that $b^4=dx_4+f_4 dx_5$. Transforming the local frame by $\exp h(0,v,0,0)$ for a suitable local function $v$ we may assume $b^4=dx_4$. Because of $db^4=0$, equation~(\ref{EI2(b)}) shows that $\bv\in I_0$. Again by~(\ref{EI2(b)}), we obtain $db^3\in I(dx_5,dx_6)$. Transforming by $\exp h\big(0,0,0,(0,y_2)^\top\big)$ for a suitable local function $y_2$ we may assume $b^3=dx_3+q_3 dx_6$. Then
\begin{gather*}
db^3 = dq_3\wedge dx_6 = -\sqrt2 \bx\wedge dx_4-\by_2\wedge dx_5+\bv\wedge (dx_6+pdx_5)\\
\hphantom{db^3 = dq_3\wedge dx_6}{} \in I\big(dx_{(4,5)}, \by_2\wedge dx_5\big)+I_1,
\end{gather*}
thus $q_3=q_3(x_5,x_6,x_7)$ and $\by_2\in I(dx_4)+I_0$. Furthermore, (\ref{EI2(b)}) yields $db^2\in I(dx_5,dx_6)$. Transforming the local frame by $\exp h\big(0,0,0,(y_1,0)^\top\big)$ for a suitable function $y_1$ we may assume $b^2=dx_2+q_2 dx_6$, thus
\begin{gather*}
db^2 = dq_2\wedge dx_6 = -\bx\wedge b^2 -\bn \wedge b^3 -\sqrt2 \bu_1\wedge dx_4-\by_1\wedge dx_5-\bv\wedge dx_7\\
\hphantom{db^2}{} \in I\big(dx_{(2,5)}, dx_{(3,5)}, dx_{(3,6)}, dx_{(4,5)}, dx_{(4,6)}, \by_1\wedge dx_5\big)+I_1.
\end{gather*}
Hence, we have $q_2=q_2(x_3,\dots,x_7)$ and $\by_1\in I(dx_2,\dots,dx_7 )$. Finally, we have $db^1\in I_0$, thus $b^1=dx_1+r_5dx_5+r_6dx_6+r_7dx_7$ and (\ref{EI2(b)}) gives
\begin{gather*}
db^1 = dr_5\wedge dx_5+dr_6\wedge dx_6+dr_7\wedge dx_7\\
\hphantom{db^1}{} \in I\big( dx_{(1,5)}, dx_{(2,5)}, dx_{(2,6)}, dx_{(3,5)}, dx_{(3,6)} \big)+dx_4\wedge I_0 +I_1.
\end{gather*}

Consequently, $r_5=r_5(x_1,\dots,x_7)$, $r_6=r_6(x_2,\dots,x_7)$, $r_7=r_7(x_4,\dots,x_7)$.

Now let the metric $g$ be defined by (\ref{Eip}) with respect to the local coordinates that we considered above. Then the holonomy of $g$ is contained in $\fh$ if and only if $b^2(\nabla b_{1})=b^3(\nabla b_{1})=b^3(\nabla b_{2})=b^3(\nabla b_{3})=0$ and
\begin{gather}
\bx=b^1(\nabla b_1)=b^2(\nabla b_2)=-b^1(\nabla b_2)= \tfrac1{\sqrt2} b^3(\nabla b_4), \label{Et2}\\
\bu_1=\ b^1(\nabla b_3)= \tfrac 1{\sqrt2}b^2(\nabla b_4), \\
\bv= b^2(\nabla b_7)= \tfrac 1{\sqrt2}b^1(\nabla b_4). \label{Et4}
\end{gather}
Note that there are six more equations which have to be satisfied, but these are automatically fulfilled with the ansatz of the adapted coframe before. Using the Koszul formula, we find $b^2(\nabla b_{1})=b^3(\nabla b_{1})=b^3(\nabla b_{2})=b^3(\nabla b_{3})=0$ and
\begin{gather*}
b^1(\nabla b_1) = (r_5)_{x_1} b^5, \\
b^1(\nabla b_2) = \big((r_5)_{x_2} -(r_6)_{x_2}p\big)b^5 + \tfrac12 \big(p_{x_6}+(r_6)_{x_2}\big) b^6, \\
b^{2}(\nabla b_{2})= \tfrac12\big((r_6)_{x_2}- p_{x_6}\big) b^5, \\
b^1(\nabla b_3)= \big((r_5)_{x_3} -(r_6)_{x_3}p\big)b^5 +\tfrac12\big((r_6)_{x_3}- (q_2)_{x_3}p\big) b^6 , \\
b^1(\nabla b_4) = \big( (r_5)_{x_4}-(r_6)_{x_4} p\big)b^5 +\tfrac12\big((r_6)_{x_4}-(q_2)_{x_4} p\big) b^6 +\frac12 (r_7)_{x_4}b^7 , \\
b^2(\nabla b_4) = \tfrac12\big((r_6)_{x_4}-(q_2)_{x_4}p\big) b^5 +(q_2)_{x_4}b^6 , \\
b^3(\nabla b_4) = \tfrac12 (r_7)_{x_4}b^5 , \\
b^2(\nabla b_7)= \tfrac12 \big((r_6)_{x_7}-(r_7)_{x_6} -(q_2)_{x_7} p+(q_3)_{x_5}\big)b^5 +(q_2)_{x_7}b^6+(q_3)_{x_7}b^7.
\end{gather*}
Hence the system (\ref{Et2})--(\ref{Et4}) is equivalent to
\begin{gather*}
 p_{x_6}=-(r_6)_{x_2},\qquad
(r_5)_{x_1} =\tfrac12\big((r_6)_{x_2}- p_{x_6}\big) =(r_6)_{x_2}p-(r_5)_{x_2}=\tfrac 1{2\sqrt2}(r_7)_{x_4}=(q_3)_{x_7},\\
(r_5)_{x_3} -(r_6)_{x_3}p= \tfrac1{2\sqrt 2}\big((r_6)_{x_4}-(q_2)_{x_4}p\big)=(q_2)_{x_7},\qquad
 (r_6)_{x_3}- (q_2)_{x_3}p= {\sqrt2}(q_2)_{x_4},\\
\sqrt2 \big( (r_5)_{x_4}-(r_6)_{x_4} p\big)=(r_6)_{x_7}-(r_7)_{x_6} -(q_2)_{x_7} p+(q_3)_{x_5},
\end{gather*}
which is equivalent to the claim. \end{proof}

\begin{Remark} \label{34}The system (\ref{Esysta})--(\ref{Esyste}) can be easily solved. Obviously, (\ref{Esysta}) and (\ref{Esyst2}) imply
\begin{gather}
(q_2)_{x_4 x_4} =2 (q_2)_{x_3 x_7}. \label{ENB}
\end{gather}
Now choose $p=p(x_5,x_6)$, $\bar r_7=\bar r_7(x_5,x_6,x_7)$, $\bar q_3=\bar q_3(x_5,x_6)$ and $q_2=q_2(x_3,\dots,x_7)$ such that $q_2$ satisfies (\ref{ENB}). By integrating~(\ref{Esysta}) and~(\ref{Esyst2}), we determine~$\bar r_6$ up to a function of~$x_5$,~$x_6$,~$x_7$, which can be chosen arbitrarily. In order to show that we can determine~$\bar r_5$ such that~(\ref{Esyst3}) and~(\ref{Esyste}) are satisfied, we have to check that the derivative of the right hand side of~(\ref{Esyst3}) with respect to~$x_4$ equals the derivative of the right hand side of~(\ref{Esyste}) with respect to~$x_3$. But this is true since
\begin{gather*} \frac{\partial}{\partial x_4} \big(((q_2)_{x_3}p+ {\sqrt2}(q_2)_{x_4})p+ (q_2)_{x_7}\big)=\big((q_2)_{x_3x_4}p+ {\sqrt2}(q_2)_{x_4x_4}\big)p+ (q_2)_{x_7x_4}\end{gather*}
and since (\ref{Esysta}) and (\ref{ENB}) imply
\begin{gather*}
 \frac{\partial}{\partial x_3}\big(\tfrac 1{\sqrt2} \big((\bar r_6)_{x_7} - (\bar r_7)_{x_6} +3 (q_2)_{x_7}p +(q_3)_{x_5}\big) +2p_{x_6 x_6} x_4 +(q_2)_{x_4} p^2\big) \\
 \qquad{} = \tfrac 1{\sqrt2}\big((\bar r_6)_{x_7x_3} +3 (q_2)_{x_7x_3}p\big) +(q_2)_{x_4x_3} p^2\\
\qquad{} = \tfrac 1{\sqrt2}\big(\sqrt2(q_2)_{x_4x_7} +4 (q_2)_{x_7x_3}p\big) +(q_2)_{x_4x_3} p^2\\
\qquad{} = (q_2)_{x_4x_7} +\sqrt2 (q_2)_{x_4x_4}p +(q_2)_{x_4x_3} p^2.
\end{gather*}
Hence we can integrate (\ref{Esyst3}) and (\ref{Esyste}). This gives $\bar r_5$ up to a function of $x_5$, $x_6$, $x_7$, which also can be chosen arbitrarily.
\end{Remark}

\begin{Remark} \label{R35} We will say that a local coframe $b$ on an open subset $U\subset \RR^7$ is {\it $\fh$-adapted} if $b$ has the form (\ref{Eb}) and if the metric $g$ given by (\ref{Eip}) has holonomy contained in $\fh$. We have shown that each pseudo-Riemannian manifold of signature (4,3) whose holonomy is contained in $\fh$ is locally isometric to some $(U,g)$, where~$U$ is an open subset of~$\RR^7$ and~$g$ is given by~(\ref{Eip}) with respect to an $\fh$-adapted coframe. Furthermore, we proved that $\fh$-adapted coframes exist and that a general $\fh$-adapted coframe depends on 2~functions in 4~variables, 3~functions in~3~variables and~2~functions in~2~variables. The latter statement follows from Remark~\ref{34}. Indeed, in the sense of Cartan--K\"ahler theory, the contact system associated with the partial differential equation $u_{zz}=u_{xy}$ (for a function $u=u(x,y,z)$) is involutive and its solution depends on~2~functions in~2~variables. Hence the general solution $q=q(x_3,\dots,x_7)$ of (\ref{ENB}) depends on 2~functions in~4~variables. Furthermore, $p$, $\bar q_3$ and $\bar r_7$ can be chosen arbitrarily, which gives~2~functions in~2~variables and 1~function in~3~variables. Finally, the determination of~$\bar r_6$ and~$\bar r_7$ as described in Remark~\ref{34} gives 2~additional arbitrary functions in~3~variables (integration constants).
\end{Remark}

\begin{Example} Starting with $p(x_5,x_6)=x_6^2$, $q_2(x_3,\dots, x_7)=x_3+x_4$, $\bar q_3=\bar r_7=0$ we obtain
\begin{gather*}
q_3(x_5,x_6,x_7) = -2x_6 x_7 ,\\
r_5(x_1,\dots, x_7) = -2x_1 x_6 +2 \big(1-x_6^2\big) x_2 x_6 +\big(x_6^2 +\sqrt2\big)x_3 x_6^2 +2x_4^2 +x_4 x_6^4,\\
r_6(x_2,\dots, x_7) = -2x_2 x_6 +\big(x_6^2 +\sqrt2\big)x_3 +x_4 x_6^2,\qquad
r_7(x_4,\dots, x_7) = -4\sqrt 2x_4 x_6 .
\end{gather*}
Then one computes
\begin{gather*}R_{56}=2h\big(\diag(1,0),0,(0,1)^\top,0\big), \qquad R_{67}=-2h(0,1,0,0), \\ (\nabla_{b_5} R)_{56}=\sqrt2 h\big({-}N,0,(1,0)^\top,0\big).\end{gather*}
Since these elements generate $\fh$ as a Lie algebra, the holonomy algebra coincides with~$\fh$.
\end{Example}

\subsection[Type I 2(c), $i=j=0$]{Type I 2(c), $\boldsymbol{i=j=0}$}
We consider
\begin{gather*}\fh=\Span\big\{X=\diag(2,1),\, h\big(N,0,(0,1)^\top,0\big) \big\}\ltimes\fm(0,0,2).\end{gather*}
The structure equations are
\begin{gather}\label{EI2(c)}\left(\begin{matrix} db^1\\db^2\\db^3\\db^4\\db^5\\db^6\\db^7\end{matrix}\right) = -\left(
\begin{matrix}
3{\bx} &-\bn&0&0&0&-\by_1&-\by_2\\
0&2\bx&\bn&0&\by_1&0&0\\
0&0&\bx&\sqrt 2 \bn&\by_2&0&0\\
0&0&0&0&0&0&\sqrt 2 \bn\\
0&0&0&0&-3\bx&0&0\\
0&0&0&0&\bn&-2\bx&0\\
0&0&0&0&0&-\bn&-\bx
\end{matrix}\right) \wedge \left(\begin{matrix} b^1\\b^2\\b^3\\b^4\\b^5\\b^6\\b^7\end{matrix}\right),
\end{gather}
where bold symbols denote 1-forms.
\begin{Proposition} \label{P36} The holonomy of $\big(M^{4,3},g\big)$ is contained in the Lie algebra $\fh$ of Type~I~$2(c)$ with $i=j=0$, if and only if we can introduce local coordinates $x_1,\dots,x_7$ such that $g=2\big(b^1\cdot b^5+b^2\cdot b^6+b^3\cdot b^7\big)- \big(b^4\big)^2$ for
\begin{gather*}
b^1= dx_1 + r_5(x_1,\dots, x_7)dx_5+r_6(x_2,x_3,x_5,x_6,x_7)dx_6+r_7(x_3,x_4,x_5,x_6,x_7)dx_7, \\
b^2= dx_2+q_2(x_3,x_5,x_6)dx_6, \qquad b^3= dx_3+q_3(x_4,x_5,x_6) dx_6, \\
b^4= dx_4+ q(x_5,x_6,x_7)dx_5+q_4(x_5,x_6,x_7)dx_6, \\
b^j= dx_j,\qquad j=5,6, \qquad
b^7= dx_7 +p(x_5,x_6,x_7)dx_5,
\end{gather*}
where $p$, $q$, $q_2$, $q_3$, $q_4$, $r_5$, $r_6$, $r_7$ are of the form
\begin{gather*}
p = ax_6x_7+bx_7+\bar p,\qquad
q = -\tfrac1{\sqrt 2} ax_7^2 -\sqrt 2 \bar p_{x_6}x_7+\bar q, \\
q_2 = 2ax_3x_6 +2bx_3+\bar q_2,\qquad
q_3= 2\sqrt 2 (ax_6+b) x_4+\bar q_3, \qquad
q_4 = 2\sqrt 2 (ax_6+b) x_7+\bar q_4, \\
r_5=-(ax_6+b)(3x_1+px_3)+(ax_7+\bar p_{x_6})\big(x_2-\sqrt2 px_4\big)+\bar r_5, \\
r_6=-4(ax_6+b)x_2-\bar p_{x_6}x_3-ax_3x_7+\bar r_6, \\
r_7=-(ax_6+b)x_3-\sqrt2 \bar p_{x_6}x_4-\sqrt2 ax_4 x_7+\bar r_7
\end{gather*}
for some functions $a=a(x_5)$, $b=b(x_5)$, $\bar p=\bar p(x_5,x_6)$, $\bar q=\bar q(x_5,x_6)$, $\bar q_j=\bar q_j(x_5,x_6)$ and $\bar r_k=\bar r_k(x_5,x_6,x_7)$ for $j=2,3,4$, $k=5,6,7$ that satisfy
\begin{gather*}
\bar p_{x_6 x_6}=2(a^2x_6^2+(2ab-a')x_6+b^2-b'), \qquad
(\bar q_4)_{x_5}-(\bar q)_{x_6}=2\sqrt 2(ax_6+b)\bar p,\\
(\bar r_6)_{x_7}-(\bar r_7)_{x_6}=(ax_6+b)\big(2ax_7^2+\bar q_3+2\sqrt2 \bar q\big)+\sqrt2 \bar q_4(ax_7+\bar p_{x_6} )-(\bar q_3)_{x_5}.
\end{gather*}
\end{Proposition}
\begin{proof} We can introduce coordinates such that $b^5\in\Span\{dx_5\}$. Changing the local frame $b_1,\dots , b_7$ by $\exp tX$ for a suitable local function $t$ we obtain $b^5=dx_5$. Then $\bx\in I(dx_5)$, hence $db^6\in I(dx_5)$ by~(\ref{EI2(c)}). Consequently, we can introduce $x_6$ such that $b^6 =dx_6 + t_1 dx_5$ for some function $t_1$. Now we change the local frame $b_1,\dots,b_7$ by $\exp \big(t_2\cdot h\big(N,0,(0,1)^\top,0\big)\big)$ for a suitable local function $t_2$ such that $b^6=dx_6$. Then $\bn\in I(dx_5,dx_6)$ by (\ref{EI2(c)}).
Thus $db^7\in I(dx_5)$, which gives $b^7=dx_7+pdx_5$ for suitable functions $x_7$ and $p$. Then
\begin{gather*}dp\wedge dx_5=db^7=\bn\wedge b^6+ \bx\wedge b^7 \in I_1\end{gather*}
implies $p=p(x_5,x_6,x_7)$. By (\ref{EI2(c)}), we get $db^4=-\sqrt 2 \bn\wedge b^7$, thus $db^4\in I(dx_5, dx_6)$. Hence we can choose $x_4$ such that $b^4=dx_4+ qdx_5+q_4dx_6$. Now
\begin{gather*}db^4=-\sqrt 2 \bn\wedge b^7\in I_1\end{gather*} gives $q=q(x_5,x_6,x_7)$ and $q_4=q_4(x_5,x_6,x_7)$. Again by (\ref{EI2(c)}), we have $db^3\in I(dx_5,dx_6)$. Thus we can choose $x_3$ such that $b^3=dx_3+t_1dx_5+t_2 dx_6$ for some functions $t_1$, $t_2$. Now we change the local frame $b_1,\dots,b_7$ by $\exp (h(0,0,0,(0,y_2)^\top))$ for a suitable local function $y_2$ such that $b^3=dx_3+q_3dx_6$. Since
\begin{gather*}dq_3\wedge dx_6=db^3=-\bx\wedge b^3 -\sqrt 2\bn\wedge b^4-\by_2\wedge b^5,\end{gather*}
we have $dq_3\in I(dx_4,dx_5, dx_6)$, thus $q_3=q_3(x_4,x_5,x_6)$. This gives $\by_2\in I(dx_3,dx_4,dx_5,dx_6)$.
Similarly, $b^2=dx_2+q_2dx_6$, where $q_2=q_2(x_3,x_5,x_6)$. Thus $\by_1\in I(dx_2,dx_3,dx_5,dx_6)$. Finally, (\ref{EI2(c)}) shows that $db^1\in I_0$. Thus we can choose $x_1$ such that $b^1=dx_1+ r_5dx_5+r_6dx_6+r_7dx_7$. Then
\begin{gather*}
 dr_5\wedge dx_5+dr_6\wedge dx_6+dr_7\wedge dx_7= -3\bx\wedge b^1+\bn\wedge b^2+\by_1\wedge b^6+\by_2\wedge b^7 \\
\qquad{} \in I\big( dx_{(1,5)}, dx_{(2,5)}, dx_{(2,6)}, dx_{(4,5)}, dx_{(4,7)}\big)+ dx_3\wedge I_0 +I_1,
\end{gather*}
hence $r_5$, $r_6$, $r_7$ are as claimed.

We proceed as in the proof of Proposition~\ref{Lg}. We compute $b^2(\nabla b_1)=b^3(\nabla b_1) =b^3(\nabla b_2) =0$ and
\begin{gather*}
b^1(\nabla b_1) = (r_5)_{x_1} b^5, \\
b^1(\nabla b_2)= (r_5)_{x_2} b^5+\tfrac12 (r_6)_{x_2} b^6,\\
b^2(\nabla b_2)= \tfrac12 (r_6)_{x_2} b^5,\\
b^1(\nabla b_3)= \big((r_5)_{x_3}- (r_7)_{x_3 }p\big) b^5+\tfrac12 \big(p_{x_6}+(r_6)_{x_3}\big) b^6+\tfrac12 \big(p_{x_7}+(r_7)_{x_3}\big) b^7,\\
 b^2(\nabla b_3)= \tfrac12 \big((r_6)_{x_3}-p_{x_6}\big) b^5 +(q_2)_{x_3}b^6,\\
b^3(\nabla b_3)= \tfrac12 \big((r_7)_{x_3}-p_{x_7}\big) b^5,\\
b^1(\nabla b_4)= \big((r_5)_{x_4}-(r_7)_{x_4}p\big)b^5-\tfrac12\big((q_4)_{x_7}p-(q_4)_{x_5}+q_{x_6} \big)b^6-\tfrac12\big(q_{x_7}-(r_7)_{x_4}\big)b^7,\\
b^2(\nabla b_4)= \tfrac12\big((q_4)_{x_7}p-(q_4)_{x_5}+q_{x_6}\big)b^5-\tfrac12\big((q_4)_{x_7}-(q_3)_{x_4}\big)b^7,\\
b^3(\nabla b_4)= \tfrac12\big(q_{x_7}+(r_7)_{x_4}\big)b^5+\tfrac12\big((q_4)_{x_7}+(q_3)_{x_4}\big)b^6,\\
b^2(\nabla b_7)= \tfrac12\big((q_3)_{x_4}-(q_4)_{x_7}\big)b^4 \\
\hphantom{b^2(\nabla b_7)=}{} +\tfrac12\big((r_6)_{x_7} -(r_7)_{x_6}+(r_7)_{x_4}q_4+(r_7)_{x_3}q_3+(q_3)_{x_5}-(q_3)_{x_4}q\big)b^5.
\end{gather*}
The equation $\bx=\frac13 b^1(\nabla b_1)=\frac12 b^2(\nabla b_2)=b^3(\nabla b_3)$ is equivalent to
\begin{gather}
\tfrac13(r_5)_{x_1} = \tfrac14(r_6)_{x_2} =\tfrac12 \big((r_7)_{x_3}-p_{x_7}\big).\label{ED1}
\end{gather}
Furthermore, $\bn=-b^1(\nabla(b_2))=b^2(\nabla b_3)=\frac1{\sqrt 2} b^3(\nabla b_4)$ is equivalent to
\begin{gather}
 -(r_5)_{x_2} =\tfrac12 \big((r_6)_{x_3}-p_{x_6}\big)=\tfrac1{2\sqrt 2} \big(q_{x_7}+(r_7)_{x_4}\big),\label{ED2}\\
- \tfrac12 (r_6)_{x_2} =(q_2)_{x_3} =\tfrac1{2\sqrt 2}\big((q_4)_{x_7}+(q_3)_{x_4}\big).\label{ED3}
\end{gather}
Finally, $ b^1(\nabla b_3) =b^1(\nabla b_4)=b^2(\nabla b_4 )=b^2(\nabla b_7)=0$ is equivalent to
\begin{gather}
 (r_7)_{x_3 }p=(r_5)_{x_3}, \label{ED4}\\
 p_{x_6}+(r_6)_{x_3}=0, \label{ED5}\\
 p_{x_7}+(r_7)_{x_3}=0, \label{ED6}\\
(r_7)_{x_4}p=(r_5)_{x_4}, \label{ED7}\\
(q_4)_{x_7}p-(q_4)_{x_5}+q_{x_6} =0, \label{ED8}\\
q_{x_7}=(r_7)_{x_4}, \label{ED9}\\
(q_4)_{x_7}=(q_3)_{x_4}, \label{ED11}\\
 (r_6)_{x_7}-(r_7)_{x_6}+(r_7)_{x_4}q_4+(r_7)_{x_3}q_3+(q_3)_{x_5}-(q_3)_{x_4}q=0.\label{ED12}
 \end{gather}

Equations (\ref{ED1}), (\ref{ED3}) and (\ref{ED6}) imply $-4 p_{x_7}=(r_6)_{x_2}=-2(q_2)_{x_3}$. Hence $p_{x_7}$ does not depend on $x_7$ and $(q_2)_{x_3}$ does not depend on $x_3$. Thus
\begin{gather*}p(x_5,x_6,x_7)=\hat p(x_5,x_6)x_7+\bar p(x_5,x_6), \qquad q_2(x_3,x_5,x_6)=2\hat p(x_5,x_6)x_3 +\bar q_2(x_5,x_6).\end{gather*}
Now (\ref{ED2}), (\ref{ED5}), (\ref{ED6}) and (\ref{ED9}) yield
\begin{gather*}
r_6 = -4 \hat px_2-\hat p_{x_6} x_3 x_7 -\bar p_{x_6} x_3 +\bar r_6,\qquad \bar r_6=\bar r_6(x_5,x_6,x_7),\\
r_7 = -\hat p x_3-\sqrt 2 \hat p_{x_6} x_4x_7-\sqrt 2\bar p_{x_6} x_4 +\bar r_7, \qquad \bar r_7=\bar r_7(x_5,x_6,x_7),
\end{gather*}
and (\ref{ED1}), (\ref{ED2}), (\ref{ED4}) and (\ref{ED7}) give
\begin{gather*} r_5= -3\hat p x_1+p_{x_6}x_2-\hat ppx_3-\sqrt2 p_{x_6}px_4+\bar r_5,\qquad \bar r_5=\bar r_5(x_5,x_6,x_7).\end{gather*}

Equations (\ref{ED11}) and (\ref{ED3}) imply
 $(q_4)_{x_7}=(q_3)_{x_4}=-\frac 1{\sqrt2} (r_6)_{x_2}=2\sqrt 2 \hat p$, thus
\begin{gather*} q_3=2\sqrt 2 \hat px_4+\bar q_3,\qquad q_4=2\sqrt 2 \hat p x_7 +\bar q_4, \qquad \bar q_j=\bar q_j(x_5,x_6),\qquad j=3,4 .\end{gather*}
Moreover, (\ref{ED9}) gives
\begin{gather*}q= -\tfrac1{\sqrt 2} \hat p_{x_6} x_7^2-\sqrt 2 \bar p_{x_6}x_7+\bar q,\qquad \bar q=\bar q(x_5,x_6). \end{gather*}

Now (\ref{ED8}) is equivalent to
\begin{gather*}\hat p_{x_6 x_6}=0,\qquad 2\hat p^2=2\hat p_{x_5}+\bar p_{x_6 x_6}, \qquad
2\sqrt 2\hat p\bar p=(\bar q_4)_{x_5}-(\bar q)_{x_6}\end{gather*}
and (\ref{ED12}) is equivalent to
\begin{gather*}(\bar r_6)_{x_7}-(\bar r_7)_{x_6}=2\hat p_{x_6}\hat px_7^2+\sqrt2\hat p_{x_6}\bar q_4x_7+\sqrt2 \bar p_{x_6} \bar q_4+\hat p\bar q_3-(\bar q_3)_{x_5}+2\sqrt2 \hat p\bar q.\end{gather*}
Putting $\hat p(x_5,x_6)=a(x_5)x_6+b(x_5)$, we obtain the assertion.
\end{proof}
\begin{Remark}\label{R38} Similarly to Remark~\ref{R35} we can determine the generality of $\fh$-adapted coframes. The functions $a$, $b$, $\bar q$, $\bar q_2$, $\bar q_3$, $\bar r_5$, $\bar r_6 $ can be chosen arbitrarily. This gives 2~functions in 3~variables, 3~functions in 2~variables and 2 functions in 1 variable. The stepwise determination of $\bar p$, $\bar q_4$ and~$\bar r_7$ gives 1 more function in 2~variables and 2~functions in 1~variable. All in all, a general $\fh$-adapted coframe depends on 2~functions in 3~variables, 4~functions in 2~variables and 4~functions in 1~variable.
\end{Remark}
\begin{Example}\label{Ex37}
For $a(x_5):=1$, $\bar p(x_5,x_6):=\frac16 x_6^4$, $\bar q_4(x_5,x_6):=\frac {\sqrt2}3x_5 x_6^5$, \begin{gather*}\bar r_6(x_5,x_6,x_7):= \tfrac23 x_6 x_7^3+\tfrac13 x_5x_6^5 x_7^2+\tfrac 49 x_5x_6^8x_7 \end{gather*}
and $b=\bar q_2=\bar q_3 =\bar q=\bar r_5= \bar r_7=0$ the holonomy equals $\fh$. To show this we determine the operators $R_{56}$, $R_{57}$ and $R_{36}$, which are in~$\fh$ by Proposition~\ref{P36}. As in Section~\ref{S3.2} we only compute those entries that are needed and write `$*$' for the remaining ones. We obtain
\begin{gather*}R_{56}=h(\diag(2,1), *,*,*),\qquad R_{57}=h(N,*,*,*,),\qquad R_{36}=h\big(0,0,0,(0,1)^\top\big),\end{gather*}
which generate $\fh$ as a Lie algebra. This proves the assertion.
\end{Example}

\subsection[Type I 2(c), $i=1$, $j=0$]{Type I 2(c), $\boldsymbol{i=1}$, $\boldsymbol{j=0}$}
We consider
\begin{gather*}\fh=\Span\big\{X=\diag(2,1),\, h\big(N,0,(0,1)^\top,0\big) \big\}\ltimes\fm(1,0,2).\end{gather*}
The structure equations are
\begin{gather}\label{EI2(c)i=1}\left(\begin{matrix} db^1\\db^2\\db^3\\db^4\\db^5\\db^6\\db^7\end{matrix}\right) = -\left(
\begin{matrix}
3{\bx} &-\bn&0&\sqrt2 \bv&0&-\by_1&-\by_2\\
0&2\bx&\bn&0&\by_1&0&\bv\\
0&0&\bx&\sqrt 2 \bn&\by_2&-\bv&0\\
0&0&0&0&\sqrt2 \bv&0&\sqrt 2 \bn\\
0&0&0&0&-3\bx&0&0\\
0&0&0&0&\bn&-2\bx&0\\
0&0&0&0&0&-\bn&-\bx
\end{matrix}\right) \wedge \left(\begin{matrix} b^1\\b^2\\b^3\\b^4\\b^5\\b^6\\b^7\end{matrix}\right).
\end{gather}
\begin{Proposition} \label{P2ci} The holonomy of $\big(M^{4,3},g\big)$ is contained in the Lie algebra $\fh$ of Type~I~$2(c)$ with $i=1$, $j=0$ if and only if we can introduce local coordinates $x_1,\dots,x_7$ such that $g=2\big(b^1\cdot b^5+b^2\cdot b^6+b^3\cdot b^7\big)- \big(b^4\big)^2$ for
\begin{gather*}
b^1 = dx_1 + r_5(x_1,\dots, x_7)dx_5+r_6(x_2,\dots,x_7)dx_6+r_7(x_3,\dots,x_7)dx_7, \\
b^2 = dx_2+q_2(x_3,x_5,x_6,x_7)dx_6, \qquad
b^3 = dx_3+q_3(x_4,\dots,x_7) dx_6, \\
b^4 = dx_4+q_4(x_5,x_6,x_7)dx_6, \qquad
b^k = dx_k,\qquad k=5,6, \\
b^7 = dx_7 +p(x_5,x_6,x_7)dx_5,
\end{gather*}
where $q_2$, $q_3$, $r_5$, $r_6$, $r_7$ are of the form
\begin{gather*}
q_2 = 2p_{x_7}x_3+\bar q_2,\qquad
q_3 = 2\sqrt 2 p_{x_7}x_4+\bar q_3,\\
r_5 = -3p_{x_7}x_1+p_{x_6} x_2-p_{x_7}px_3 + \big(2 p_{x_5x_7}+p_{x_6 x_6}-2p_{x_7}^2\big)x_4^2 \\
 \hphantom{r_5 =}{} + \tfrac1{\sqrt2} \big((\bar r_6)_{x_7}-(\bar r_7)_{x_6}-2\sqrt2 p_{x_6}q_4-p_{x_7}\bar q_3 -3p_{x_6}p+(\bar q_3)_{x_5}\big)x_4+\bar r_5,\\
r_6 = -4p_{x_7}x_2- p_{x_6}x_3+\sqrt2 (\bar q_2)_{x_7}x_4+\bar r_6,\qquad
r_7 = -p_{x_7}x_3-2\sqrt2 p_{x_6}x_4+\bar r_7
\end{gather*}
for functions $\bar q_k=\bar q_k(x_5,x_6,x_7)$, $k=2,3$, $\bar r_l=\bar r_l(x_5,x_6,x_7)$, $l=5,6,7$, satisfying
\begin{gather*}p_{x_7 x_7}=0, \qquad (\bar q_3)_{x_7}=-p_{x_6},\qquad (q_4)_{x_7}=2\sqrt2 p_{x_7}, \qquad
\sqrt2 (\bar q_2)_{x_7}=(q_4)_{x_5}-2\sqrt2 p_{x_7}p.
\end{gather*}
\end{Proposition}
\begin{proof} As for Type I 2(c), $i=j=0$, we can introduce coordinates such that $b^k=dx_k$, $k=5,6$, $b^7= dx_7 +p(x_5,x_6,x_7)dx_5$ and	 $\bx\in I(dx_5)$, $\bn\in I(dx_5,dx_6)$. But now we can assume additionally, that $b^4=dx_4+q_4 dx_6$. Then
\begin{gather*}db^4=dq_4\wedge dx_6=-\sqrt2\bv\wedge dx_5-\sqrt 2 \bn\wedge b^7\end{gather*} gives $q_4=q_4(x_5,x_6,x_7)$ and $\bv\in I_0$.
Again by~(\ref{EI2(c)i=1}), we have $db^3\in I(dx_5,dx_6)$. As for Type~I~2(c), $i=j=0$ we can choose~$x_3$ such that $b^3=dx_3+q_3dx_6$. Since
\begin{gather*}
dq_3\wedge dx_6 = db^3=-\bx\wedge b^3 -\sqrt 2\bn\wedge b^4-\by_2\wedge dx_5+\bv\wedge dx_6\\
\hphantom{dq_3\wedge dx_6}{} \in I\big(dx_{(3,5)}, dx_{(4,5)}, dx_{(4,6)}, \by_2\wedge dx_5\big)+I_1,
\end{gather*}
we have $q_3=q_3(x_4,\dots,x_7)$ and $\by_2\in I(dx_3,\dots,dx_7)$.
Similarly, $b^2=dx_2+q_2dx_6$, where $q_2=q_2(x_3,x_5,x_6,x_7)$. Thus $\by_1\in I(dx_2,dx_3)+I_0$. Finally, (\ref{EI2(c)i=1}) shows that $db^1\in I_0$. Thus we can choose $b^1=dx_1+ r_5dx_5+r_6dx_6+r_7dx_7$. Then
\begin{gather*}
 dr_5\wedge dx_5+dr_6\wedge dx_6+dr_7\wedge dx_7= -3\bx\wedge b^1+\bn\wedge b^2-\sqrt2\bv\wedge b^4+\by_1\wedge b^6+\by_2\wedge b^7 \\
 \qquad{}\in I\big( dx_{(1,5)}, dx_{(2,5)}, dx_{(2,6)}\big)+ dx_3\wedge I_0 +dx_4\wedge I_0+I_1,
\end{gather*}
hence $r_5$, $r_6$, $r_7$ are as claimed. We proceed as in the proof of Proposition~\ref{Lg}. For $\nabla b_1=\nabla b_2=\nabla b_3 =0$, we get the same formulas as in the case of Type~I~2(c), $i=j=0$. Furthermore,
\begin{gather*}
b^1(\nabla b_4)= \big((r_5)_{x_4}-(r_7)_{x_4}p\big)b^5+\tfrac12\big({-}(q_4)_{x_7}p+(q_4)_{x_5}+(r_6)_{x_4} \big)b^6 + \tfrac12(r_7)_{x_4}b^7,\\
b^2(\nabla b_4)= \tfrac12\big((q_4)_{x_7}p-(q_4)_{x_5}+(r_6)_{x_4}\big)b^5 +\tfrac12\big((q_3)_{x_4}-(q_4)_{x_7}\big)b^7,\\
b^3(\nabla b_4)= \tfrac12(r_7)_{x_4}b^5+\tfrac12\big((q_4)_{x_7}+(q_3)_{x_4}\big)b^6,\\
b^2(\nabla b_7)= \tfrac12((q_3)_{x_4}-(q_4)_{x_7})b^4 +\tfrac12\big((r_6)_{x_7} -(r_7)_{x_6}+(r_7)_{x_4}q_4+(r_7)_{x_3}q_3+(q_3)_{x_5} \\
\hphantom{b^2(\nabla b_7)=}{} -(q_3)_{x_7}p\big)b^5 +(q_2)_{x_7}b^6 +(q_3)_{x_7}b^7.
\end{gather*}
The equation $\bx=\frac13 b^1(\nabla b_1)=\frac12 b^2(\nabla b_2)=b^3(\nabla b_3)$ is equivalent to (\ref{ED1}). Furthermore,
$\bn=-b^1(\nabla(b_2))=b^2(\nabla b_3)=\frac1{\sqrt 2} b^3(\nabla b_4)$ is equivalent to
\begin{gather}
 -(r_5)_{x_2} =\tfrac12 \big((r_6)_{x_3}-p_{x_6}\big)=\tfrac1{2\sqrt 2} (r_7)_{x_4},\label{ED2a}\\
- \tfrac12 (r_6)_{x_2} =(q_2)_{x_3} =\tfrac1{2\sqrt 2}\big((q_4)_{x_7}+(q_3)_{x_4}\big),
\end{gather}
and $\bv=\frac1{\sqrt2} b^1(\nabla b_4)=b^2(\nabla b_7)$ is equivalent to
\begin{gather}
(q_3)_{x_4}=(q_4)_{x_7},\\
\sqrt2\big((r_5)_{x_4}-(r_7)_{x_4}p\big)=(r_6)_{x_7} -(r_7)_{x_6}+(r_7)_{x_4}q_4+(r_7)_{x_3}q_3+(q_3)_{x_5}-(q_3)_{x_7}p,\\
-(q_4)_{x_7}p+(q_4)_{x_5}+(r_6)_{x_4} = 2\sqrt2 (q_2)_{x_7}, \\
(r_7)_{x_4} = 2\sqrt2 (q_3)_{x_7}.
\end{gather}
Finally, $ b^1(\nabla b_3) =b^2(\nabla b_4 )=0$ is equivalent to
\begin{gather}
(r_5)_{x_3}=(r_7)_{x_3 }p,\qquad p_{x_6}+(r_6)_{x_3}=0,\qquad p_{x_7}+(r_7)_{x_3}=0, \\
(q_4)_{x_7}p-(q_4)_{x_5}+(r_6)_{x_4}=0, \qquad (q_3)_{x_4}=(q_4)_{x_7} . \label{ED3a}
 \end{gather}
Equations (\ref{ED1}) and (\ref{ED2a})--(\ref{ED3a}) simplify to
\begin{gather*}
p_{x_6}=-(q_3)_{x_7}=(r_5)_{x_2}=-(r_6)_{x_3}=- \tfrac1{2\sqrt2} (r_7)_{x_4},\\
p_{x_7}= \tfrac 12 (q_2)_{x_3}=\tfrac1{2\sqrt2}(q_3)_{x_4}=\tfrac1{2\sqrt2}(q_4)_{x_7}=-\tfrac13(r_5)_{x_1}=-\tfrac14(r_6)_{x_2}=-(r_7)_{x_3},\\
(r_5)_{x_3}=- p_{x_7}p,\qquad
(r_6)_{x_4}=(q_4)_{x_5}-2\sqrt2 p_{x_7}p=\sqrt2 (q_2)_{x_7},\\
\sqrt2 (r_5)_{x_4}=(r_6)_{x_7}-(r_7)_{x_6}-2\sqrt2 p_{x_6}q_4-p_{x_7}q_3-3p_{x_6}p+(q_3)_{x_5}.
\end{gather*}
These equations imply, in particular, $2p_{x_7 x_7}=(q_2)_{x_3 x_7}=(q_2)_{x_7 x_3}=((q_4)_{x_5}-2\sqrt2 p_{x_7}p)_{x_3}=0$. Now the second part of the proposition follows easily.
\end{proof}
\begin{Remark} In the same way as in Remarks~\ref{R35} and~\ref{R38} we determine the generality of $\fh$-adapted coframes. The functions $ \bar r_5$, $\bar r_6$, $\bar r_7$ can be chosen arbitrarily. This gives 3~functions in 3~variables. Integration of $p_{x_7 x_7}=0$ shows that the choice of $p$ depends on 2~functions in 2~variables. Then $\bar q_3$, $q_4$ and $\bar q_2$ depend on 3 arbitrary functions in 2~variables. Thus a general $\fh$-adapted coframe depends on 3 functions in 3 variables and 5 functions in 2 variables.
\end{Remark}

\begin{Example}
Starting with $p(x_5,x_6,x_7)=x_6x_7$, $\bar q_2(x_5,x_6,x_7)=-x_6^2x_7^2$, $\bar q_3(x_5,x_6,x_7)= x_5-\frac 12x_7^2$, $q_4(x_5,x_6,x_7)=2\sqrt2 x_6 x_7$ and $\bar r_5 =\bar r_6=\bar r_7=0$ we get
\begin{gather*}
q_2= 2x_3 x_6-x_6^2x_7^2, \qquad
q_3= 2\sqrt2 x_4 x_6 +x_5- \tfrac12 x_7^2, \\
r_5= -3x_1 x_6+x_2 x_7 -x_3 x_6^2x_7 -2 x_4^2x_6^2 +\tfrac1{\sqrt2}\big(1-x_5x_6-\tfrac{21}2 x_6 x_7^2\big) x_4, \\
r_6= -4x_2 x_6-x_3x_7-2\sqrt2 x_4 x_6^2 x_7, \qquad
r_7= -x_3x_6-2\sqrt2 x_4 x_7 .
\end{gather*}
We want to show that the holonomy equals $\fh$. We proceed as in Example~\ref{Ex37} and compute the following parts of the curvature tensor and its covariant derivative:
\begin{gather*}
R_{56}=h(\diag(2,1), *,*,*),\qquad R_{57}=h(N,*,*,*,),\qquad R_{25}=-h\big(0,0,0,(0,1)^\top\big),\\
(\nabla_{b_5} R)_{56}=-h(0,1,0,*),
\end{gather*}
which generate $\fh$ as a Lie algebra.
\end{Example}

\subsection[Type I 2(c), $i=j=1$]{Type I 2(c), $\boldsymbol{i=j=1}$}
We consider
\begin{gather*}\fh=\Span\big\{X=\diag(2,1), \, h\big(N,0,(0,1)^\top,0\big) \big\}\ltimes\fm(1,1,2).\end{gather*}
The structure equations are
\begin{gather}\label{EI2(c)ij=1}\left(\begin{matrix} db^1\\db^2\\db^3\\db^4\\db^5\\db^6\\db^7\end{matrix}\right) = -\left(
\begin{matrix}
3{\bx} &-\bn&\bu_1&\sqrt2 \bv&0&-\by_1&-\by_2\\
0&2\bx&\bn&\sqrt2 \bu_1&\by_1&0&\bv\\
0&0&\bx&\sqrt 2 \bn&\by_2&-\bv&0\\
0&0&0&0&\sqrt2 \bv&\sqrt2 \bu_1&\sqrt 2 \bn\\
0&0&0&0&-3\bx&0&0\\
0&0&0&0&\bn&-2\bx&0\\
0&0&0&0&-\bu_1&-\bn&-\bx
\end{matrix}\right) \wedge \left(\begin{matrix} b^1\\b^2\\b^3\\b^4\\b^5\\b^6\\b^7\end{matrix}\right).
\end{gather}
\begin{Proposition} The holonomy of $\big(M^{4,3},g\big)$ is contained in the Lie algebra $\fh$ of Type~I~$2(c)$, $i=j=1$ if and only if we can introduce coordinates $(x_1,\dots, x_7)$ such that $g=2\big(b^1\cdot b^5+b^2\cdot b^6+b^3\cdot b^7\big)- \big(b^4\big)^2$ for
\begin{gather*}
b^1= dx_1 + r_5(x_1,\dots, x_7)dx_5+r_6(x_2,\dots,x_7)dx_6+r_7(x_3,\dots,x_7)dx_7, \\
b^2= dx_2+q_2(x_3,\dots,x_7)dx_6 +s(x_4,\dots,x_7)dx_7, \\
b^3= dx_3+q_3(x_4,\dots,x_7) dx_6, \qquad
b^4= dx_4+q_4(x_5,x_6,x_7)dx_6, \\
b^k= dx_k,\qquad k=5,6,7 ,
\end{gather*}
where $q_2$, $q_3$, $s$, $r_5$, $r_6$, $r_7$ are of the form
\begin{gather*}
q_2 = \tfrac {\sqrt2} 3ax_3+\tfrac {\sqrt2} 3 a_{ x_7}x_4^2 +\sqrt2 b x_4+\bar q_2,\qquad
q_3 = \tfrac 23 a x_4+\bar q_3,\qquad
s = -\tfrac13 a x_4+\bar s, \\
r_5 = -\tfrac 1{\sqrt 2} a x_1 -\tfrac 23 a_{x_7} x_2 x_4 -b x_2 +\tfrac 1{3\sqrt2} a_{x_7} x_3^2 +\tfrac{\sqrt2}3 a_{x_7 x_7 }x_3 x_4^2+\big( \tfrac13 a_{x_6} +\sqrt2 b_{x_7}\big) x_3 x_4 \\
\hphantom{r_5 =}{} -\big( \tfrac1{\sqrt2} (q_4)_{x_5} + \bar s_{x_6} - (\bar q_2)_{x_7} +\tfrac13 a q_4\big) x_3 +\tfrac1{9\sqrt2} a_{x_7 x_7 x_7} x_4^4 \\
\hphantom{r_5 =}{}+\big(\tfrac{\sqrt2}3 b_{x_7 x_7}-\tfrac19 a_{x_6 x_7}\big) x_4^3 +\big( {-}\bar s_{x_6 x_7}+(\bar q_2)_{x_7 x_7} +\tfrac13 a_{x_7}q_4 -\tfrac59 a^2 -b_{x_6}\big) x_4^2 \\
\hphantom{r_5 =}{} +\big( \tfrac1{\sqrt2} (\bar r_6)_{x_7}-(\bar r_7)_{x_6} -\bar s_5 +(\bar q_3)_{x_5}) +2b q_4 -\tfrac13 \bar q_3 a +\tfrac23 \bar s\big) x_4 +\bar r_5,\\
r_6=-\tfrac{2\sqrt2}3 a x_2 +\tfrac43 a_{x_7} x_3 x_4 +2b x_3 +\tfrac49 a_{x_7 x_7} x_4^3 +\big(\tfrac{\sqrt2}3 a_{x_6} +2b_{x_7}\big) x_4^2 \\
\hphantom{r_6=}{}+\big( 2\sqrt2 (-\bar s_{x_6} +(\bar q_2)_{x_7}-\tfrac13 a q_4) -(q_4)_{x_5}\big) x_4 + \bar r_6,\\
r_7=-\tfrac{\sqrt2}3 a x_3 +\tfrac{2\sqrt2}3 a_{x_7} x_4^2 +2\sqrt2 b x_4 +\bar r_7
\end{gather*}
for arbitrary functions $\bar s=\bar s(x_5,x_6,x_7)$, $\bar q_k=\bar q_k(x_5,x_6,x_7)$, $k=2,3$, $q_4=q_4(x_5,x_6,x_7)$, $\bar r_l=\bar r_l(x_5,x_6,x_7)$, $l=5,6,7$ and $a:= (q_4)_{x_7}$, $b:= (\bar q_3)_{x_7}$.
\end{Proposition}
\begin{proof} We proceed as in the proof of Proposition~\ref{P2ci} and obtain now $b^k=dx_k$, $k=5,6,7$, and $\bx\in I(dx_5)$, $\bn\in I(dx_5,dx_6)$, $\bu_1\in I_0$. Moreover, as above, $b^4=dx_4+q_4 dx_6$, which implies
\begin{gather*}db^4=dq_4\wedge dx_6=\sqrt2 (-\bv\wedge dx_5-\bu_1\wedge dx_6 -\bn\wedge dx_7).\end{gather*}
Hence $q_4=q_4(x_5,x_6,x_7)$ and $\bv\in I_0$. Furthermore, we have $db^3\in I(dx_5,dx_6)$ and can choose $x_3$ such that $b^3=dx_3+q_3dx_6$. As in the proof of Proposition~\ref{P2ci}, we see that $q_3=q_3(x_4,\dots,x_7)$ and $\by_2\in I(dx_3,\dots,dx_7)$. Now (\ref{EI2(c)ij=1}) gives $db^2\in I_0$. Thus $b^2=dx_2+q_2dx_6+sdx_7$. Since
\begin{gather*}
dq_2\wedge dx_6+ds\wedge dx_7 = db^2=-2\bx\wedge b^2 -\bn\wedge b^3-\sqrt2 \bu_1\wedge b^4-\by_1\wedge dx_5+\bv\wedge dx_7\\
\hphantom{dq_2\wedge dx_6+ds\wedge dx_7}{} \in I\big(dx_{(2,5)}, dx_{(3,5)}, dx_{(3,6)}, dx_{(4,5)}, dx_{(4,6)},dx_{(4,7)}, \by_1\wedge dx_5\big)+I_1,
\end{gather*}
we obtain $q_2=q_2(x_3,\dots,x_7)$, $s=s(x_4,\dots,x_7)$ and $\by_1\in I(dx_2,\dots,dx_7)$. Finally, (\ref{EI2(c)ij=1}) shows that $db^1\in I_0$. Thus we can choose $b^1=dx_1+ r_5dx_5+r_6dx_6+r_7dx_7$. Then $db^1$ equals
\begin{gather*}
dr_5\wedge dx_5+dr_6\wedge dx_6+dr_7\wedge dx_7\in I( dx_{(1,5)}, dx_{(2,5)}, dx_{(2,6)})+ dx_3\wedge I_0 +dx_4\wedge I_0+I_1,
\end{gather*}
hence $r_5$, $r_6$, $r_7$ are as claimed.

We proceed as in the proof of Proposition~\ref{Lg}. The equation $\bx=\frac13 b^1(\nabla b_1)=\frac12 b^2(\nabla b_2)=b^3(\nabla b_3)$ is equivalent to
\begin{gather}
\tfrac13(r_5)_{x_1} = \tfrac14(r_6)_{x_2} =\tfrac12 (r_7)_{x_3}.\label{ED1b}
\end{gather}
Furthermore, $\bn=-b^1(\nabla(b_2))=b^2(\nabla b_3)=\frac1{\sqrt 2} b^3(\nabla b_4)$ is equivalent to
\begin{gather}
 -(r_5)_{x_2} =\tfrac12 (r_6)_{x_3}=\tfrac1{2\sqrt 2} (r_7)_{x_4},\label{ED2b}\\
- \tfrac12 (r_6)_{x_2} =(q_2)_{x_3} =\tfrac1{2\sqrt 2}\big((q_4)_{x_7}+s_{x_4}+(q_3)_{x_4}\big), \label{E60}
\end{gather}
and $\bu_1= b^1(\nabla b_3) =\frac 1{\sqrt2}b^2(\nabla b_4 )$ is equivalent to
\begin{gather}
 -2\sqrt2 (r_5)_{x_3}=(q_4)_{x_5} -(r_6)_{x_4},\label{E61}\\
 (r_6)_{x_3}=\sqrt2 (q_2)_{x_4}, \label{E62}\\
 \sqrt2(r_7)_{x_3}= -(q_4)_{x_7} +s_{x_4}+(q_3)_{x_4} . \label{ED3b}
 \end{gather}

Finally, $\bv=\frac1{\sqrt2} b^1(\nabla b_4)=b^2(\nabla b_7)$ is equivalent to
\begin{gather}
 (q_3)_{x_4}=(q_4)_{x_7}+s_{x_4}, \label{E64}\\
 \sqrt2(r_5)_{x_4}=(r_6)_{x_7} -(r_7)_{x_6}+(r_7)_{x_4}q_4+(r_7)_{x_3}q_3-(r_6)_{x_2}s - s_{x_5}+(q_3)_{x_5},\label{E65}\\
 (q_4)_{x_5}+(r_6)_{x_4} = 2\sqrt2 \big({-}s_{x_6}+ (q_2)_{x_7} +s_{x_4}q_4\big), \label{E66} \\
 (r_7)_{x_4} = 2\sqrt2 (q_3)_{x_7} . \label{E67}
\end{gather}
Equations (\ref{ED1b}), (\ref{E60}), (\ref{ED3b}) and (\ref{E64}) are equivalent to
\begin{gather}\label{E1'}
\tfrac13 (r_5)_{x_1} = \tfrac14 (r_6)_{x_2}=\tfrac12 (r_7)_{x_3} =-\tfrac12 (q_2)_{x_3}=-\tfrac 1{2\sqrt 2} (q_3)_{x_4}=-\tfrac1{3\sqrt2} (q_4)_{x_7}=\tfrac 1{\sqrt2} s_{x_4}.
\end{gather}
Equations (\ref{ED2b}), (\ref{E62}) and (\ref{E67}) are equivalent to
\begin{gather}\label{E*}
-(r_5)_{x_2}=\tfrac12 (r_6)_{x_3}=\tfrac1{\sqrt 2}(q_2)_{x_4}=\tfrac1 {2\sqrt2}(r_7)_{x_4} =(q_3)_{x_7} .
\end{gather}
Hence we have to solve the system consisting of equations (\ref{E61}), (\ref{E65}), (\ref{E66}), (\ref{E1'}) and (\ref{E*}). This is done in a straightforward way starting with (\ref{E1'}) and (\ref{E*}), continuing with (\ref{E66}) and finishing with (\ref{E61}) and (\ref{E65}).
The result shows that $q_k$, $k=2,3,4$, $s$ and $r_l$, $l=5,6,7$ have the claimed form. \end{proof}

\begin{Remark} Proceeding similarly as in the previous cases we see that a general $\fh$-adapted coframe depends on 7~functions in 3~variables.
\end{Remark}

\begin{Example} Choosing $q_4(x_5,x_6,x_7)=x_6 x_7$, $\bar q_3(x_5,x_6,x_7)=x_6 x_7+\frac {\sqrt2}3 x_7^2$ and $\bar s=\bar q_2=\bar r_5 =\bar r_6 =\bar r_7=0$ we obtain
\begin{gather*}
b(x_6,x_7) = x_6+\tfrac{2\sqrt2}3 x_7,\qquad
s(x_4,\dots,x_7)= -\tfrac13 x_4x_6,\\
q_2(x_3,\dots,x_7)= \tfrac{ \sqrt2}3 x_3 x_6 +\sqrt2 b(x_6,x_7)x_4 ,\\
r_5(x_1,\dots,x_7)= -\tfrac1{\sqrt2} x_1 x_6-b(x_6 ,x_7)x_2 +\tfrac 53 x_3 x_4 -\tfrac 13 x_3 x_6^2 x_7-\tfrac59 x_4^2 x_6^2 -x_4^2 \\
\hphantom{r_5(x_1,\dots,x_7)=}{} +\tfrac 53 x_4 x_6^2 x_7 +\tfrac {11}9 \sqrt2 x_4 x_6 x_7^2,\\
r_6(x_2,\dots,x_7)= -\tfrac{2\sqrt2}3x_2 x_6 +2b(x_6, x_7)x_3 +\tfrac {5\sqrt2}3 x_4^2 -\tfrac{2\sqrt2}3x_4 x_6^2 x_7, \\
r_7(x_3,\dots,x_7)= -\tfrac{\sqrt2}3 x_3 x_6 +2\sqrt2 b(x_6, x_7) x_4.
\end{gather*}
One computes
\begin{gather*} R_{56}=\tfrac 1{3\sqrt2} X-h\big(N,0,(0,1)^\top,0\big),\qquad R_{57}=-\tfrac{2\sqrt2}3 h\big(N,0,(0,1)^\top,0\big),\\ R_{45}=h\big(0,-\sqrt2,\big(\tfrac53,0\big)^\top,0\big). \end{gather*}
These operators generate $\fh$ as a Lie algebra. Hence the holonomy is equal to $\fh$.
\end{Example}

\subsection{Type I 3(b)}
We consider
\begin{gather*}\fh=\Span\big\{ h\big(\diag(1,0),\, 0,(0,1)^\top,0\big) \big\}\ltimes\fm(1,1,2).\end{gather*}
The structure equations are
\begin{gather}\label{EI3(b)}\left(\begin{matrix} db^1\\db^2\\db^3\\db^4\\db^5\\db^6\\db^7\end{matrix}\right) = -\left(
\begin{matrix}
{\bf a}_1&- {\bf a}_1&{\bf u}_1&\sqrt{2} {\bf v}&0&-\by_1&-\by_2\\
0&{\bf a}_1&0&\sqrt{2} {\bf u}_1&\by_1&0&{\bf v}\\
0&0&0& \sqrt{2} {\bf a}_1&\by_2&-{\bf v}&0\\
0&0&0&0&\sqrt{2} {\bf v}&\sqrt{2} {\bf u}_1&\sqrt 2 {\bf a}_1\\
0&0&0&0&-{\bf a}_1&0&0\\
0&0&0&0&{\bf a}_1&-{\bf a}_1&0\\
0&0&0&0&-{\bf u}_1&0&0
\end{matrix}\right) \wedge \left(\begin{matrix} b^1\\b^2\\b^3\\b^4\\b^5\\b^6\\b^7\end{matrix}\right),
\end{gather}
where bold symbols denote 1-forms.
\begin{Proposition} The holonomy of $\big(M^{4,3},g\big)$ is contained in the Lie algebra $\fh$ of Type~I~$3(b)$ if and only if there are local coordinates $x_1,\dots,x_7$ such that $g=2\big(b^1\cdot b^5+b^2\cdot b^6+b^3\cdot b^7\big)- \big(b^4\big)^2$ for
\begin{gather*}
b^1= dx_1 + r_5(x_1,\dots, x_7)dx_5+r_6(x_2,x_4,x_5, x_6,x_7)dx_6+r_7(x_4,x_5,x_6,x_7)dx_7, \\
b^i= dx_i+q_i(x_5,x_6,x_7)dx_6,\qquad i=2,3, \\
b^6= dx_6 + p(x_5,x_6) dx_5 ,\qquad
b^j= dx_j,\qquad j=4,5,7,
\end{gather*}
where
$p=p(x_5,x_6)$ and $q_2=q_2(x_5,x_6,x_7)$ are arbitrary and the functions $q_3$, $r_5$, $r_6$, $r_7$ are of the form
\begin{gather*}
q_3 = -p_{x_6}x_7+\bar q_3,\\
r_5 = -p_{x_6} x_1 +p_{x_6}(1-p)x_2 +(q_2)_{x_7}x_3+\tfrac1{\sqrt2}\big(3(q_2)_{x_7}p+(q_3)_{x_5}+(\hat r_6)_{x_7}-(\hat r_7)_{x_6}\big)x_4 \\
\hphantom{r_5 =}{} +\big((q_2)_{x_7 x_7}+p_{x_6 x_6}\big)x_4^2 +\hat r_5, \\
r_6 = -p_{x_6}x_2 +2\sqrt2 (q_2)_{x_7} x_4+\hat r_6,\qquad
r_7 = -2\sqrt2 p_{x_6}x_4 +\hat r_7.
\end{gather*}
for arbitrary functions $\bar q_3=\bar q_3(x_5,x_6)$ and $\hat r _j =\hat r_j(x_5,x_6,x_7)$, $j=5,6,7$.
\end{Proposition}

\begin{proof} We can introduce coordinates such that $b^5\in\Span\{dx_5\}$. Transforming the local frame by $\exp h\big(\diag(x,0),0,(0,x)^\top,0\big)$ for a suitable local function $x$ we may assume $b^5=dx_5$. 	 Then ${\bf a}_ 1 \in I(dx_5)$, hence $db^6\in I(dx_5)$ by~(\ref{EI3(b)}). Thus we can introduce $x_6$ such that $b^6=dx_6 +p dx_5$. Hence $p=p(x_5,x_6)$. Furthermore, (\ref{EI3(b)}) shows $db^7\in I(dx_5)$. Hence we can introduce~$x_7$ such that $b^7=dx_7+f_5 dx_5$. Transforming the local frame by $\exp h(0,(u_1,0)^\top,0,0)$ for a suitable local function~$u_1$ we may assume $b^7=dx_7$. By
\begin{gather*}
0 = d b^7 = {\bf u}_1 \wedge b^5,
\end{gather*}
we obtain that ${\bf u}_1 \in I (dx_5)$. Consequently, $db^4 \in I(dx_5)$. Hence we can introduce $x_4$ such that $b^4 = dx_4 + f_4 dx_5$. Transforming the local frame by $\exp h(0, v, 0, 0)$ for a suitable local function~$v$ we may assume $b^4 = dx_4$. By
\begin{gather*}
0 = d b^4 = - \sqrt{2} {\bf v} \wedge b^5 - \sqrt{2} {\bf u}_1 \wedge b^6 - \sqrt{2} {\bf a}_1 \wedge b^7,
\end{gather*}
we obtain that ${\bf v} \in I_0$. Again by (\ref{EI3(b)}), we obtain $db^3 \in I(dx_5,dx_6)$. Transforming the local frame by $\exp h\big(0,0,0,(0,y_2)^\top\big)$ for a suitable function $y_2$ we may assume $b^3 = dx_3 + q_3 dx_6$. Then
\begin{gather*}
db^3 = dq_3\wedge dx_6= -\sqrt2 {\bf a}_1 \wedge dx_4-\by_2\wedge dx_5+\bv\wedge (dx_6+pdx_5)
\in I\big(dx_{(4,5)}, \by_2\wedge dx_5\big)+I_1,\!
\end{gather*}
thus $q_3=q_3(x_5,x_6,x_7)$ and $\by_2\in I(dx_4)+I_0$. Furthermore, (\ref{EI3(b)}) yields $db^2\in I(dx_5,dx_6)$. Transforming the local frame by $\exp h\big(0,0,0,(y_1,0)^\top\big)$ for a suitable function $y_1$ we may assume $b^2=dx_2+q_2 dx_6$, thus
\begin{gather*}
db^2 = dq_2\wedge dx_6 = -{\bf a}_ 1 \wedge b^2 -\sqrt2 \bu_1\wedge dx_4-\by_1\wedge dx_5-\bv\wedge dx_7\\
\hphantom{db^2 = dq_2\wedge dx_6}{} \in I\big(dx_{(2,5)}, dx_{(4,5)}, \by_1\wedge dx_5\big)+I_1.
\end{gather*}
Hence, we get $q_2=q_2(x_5,x_6,x_7)$ and $\by_1\in I(dx_2, dx_4)+I_0$. Finally, we have $db^1\in I_0$, thus $b^1=dx_1+r_5dx_5+r_6dx_6+r_7dx_7$ and (\ref{EI3(b)}) gives
\begin{gather*}
db^1 = r_5\wedge dx_5+r_6\wedge dx_6+r_7\wedge dx_7
 \in I( dx_{(1,5)}, dx_{(2,5)}, dx_{(2,6)}, dx_{(3,5)} )+dx_4\wedge I_0 +I_1.
\end{gather*}

Consequently, $r_5=r_5(x_1,\dots,x_7)$, $r_6=r_6(x_2, x_4, x_5, x_6, x_7)$, $r_7=r_7(x_4,\dots,x_7)$.

Now let the metric $g$ be defined by (\ref{Eip}) with respect to the local coordinates that we considered above. Since the expression for $b^i$ in the local coordinates $(x_1, \ldots, x_7)$ is the same as for the case Type~I~2(b) with $r_6 = r_6 (x_2, x_4, \ldots, x_7)$ and $q_2 = q_2(x_5,x_6,x_7)$, we can proceed as in the proof of Proposition \ref{prop2.2} by imposing the extra conditions coming from \begin{gather*}b^2 (\nabla b_3) = \tfrac{1}{2} \big( (q_2)_{x_3} p - (r_6)_{x_3} \big) b^5 + (q_2)_{x_3} b^6 =0,\end{gather*} and $ (r_6)_{x_3} =(q_2)_{x_3} = (q_2)_{x_4} =0$. From equations (\ref{Esysta})--(\ref{Esyste}) we now obtain the system
\begin{gather*}
(\bar r_6)_{x_4}=2\sqrt 2(q_2)_{x_7}, \qquad
(\bar r_5)_{x_3} = (q_2)_{x_7}, \\
(\bar r_5)_{x_4}=\tfrac 1{\sqrt2}\big((\bar r_6)_{x_7} - (\bar r_7)_{x_6} +3 (q_2)_{x_7}p +(q_3)_{x_5}\big) +2p_{x_6 x_6} x_4.
\end{gather*}
By integrating the first equation we determine $\bar r_6$ up to an arbitrary function $\hat r_6(x_5, x_6,x_7)$ and, by integrating the other two equations, we get the function $\bar r_5$, up to an arbitrary function $\hat r_5(x_5, x_6,x_7)$.
 \end{proof}

\begin{Remark} Proceeding similarly as in the previous cases we see that in this case a general $\fh$-adapted coframe depends on 4~functions in 3~variables and 2~functions in 2~variables.
\end{Remark}
\begin{Example} For \begin{gather*}
p(x_5,x_6)= x_5 x_6^2, \qquad
q_2(x_5,x_6, x_7) = x_7+x_6 x_7,\qquad
q_3(x_5,x_6,x_7) = -2 x_5 x_6 x_7 ,\\
r_5(x_1,\dots, x_7) = -2x_1 x_5 x_6 +2 x_2 x_5 x_6 \big(1- x_5x_6^2\big) + (1 + x_6) x_3 \\
\hphantom{r_5(x_1,\dots, x_7) =}{} + \tfrac1{\sqrt2}\big(3 ( x_6 +1) x_5 x_6^2 - 2 x_6 x_7 \big) x_4 + 2 x_4^2 x_5,\\
r_6(x_2, x_4, \dots, x_7) = -2x_2 x_5 x_6 + 2 \sqrt{2}(x_6 + 1) x_4,\\
r_7(x_4,\dots, x_7) = -4\sqrt 2 x_4 x_5 x_6
\end{gather*}
the holonomy equals $\frak h$. Indeed, the elements
\begin{gather*}
R_{56}=h\big(0,0, (-1,0)^\top,(3,0)^\top\big),\qquad (\nabla_{b_4} R)_{56}= h\big(0,0,0,(0,\sqrt{2})^\top\big),\\
(\nabla_{b_6} R)_{57}= h(0,-1,0,0), \qquad (\nabla_{b_5} R)_{56}= h\big(\diag(2,0),0,(0,2)^\top,0\big)
\end{gather*}
 generate $\fh$ as a Lie algebra, the holonomy algebra coincides with~$\fh$.
\end{Example}

\subsection[Type I 4(b), $j=0$]{Type I 4(b), $\boldsymbol{j=0}$}
We consider
\begin{gather*}\fh=\RR\cdot h\big(N,0,(0,1)^\top,0\big) \ltimes\fm(1,0,2).\end{gather*}

The structure equations are
\begin{gather}\label{EI4(b)}\left(\begin{matrix} db^1\\db^2\\db^3\\db^4\\db^5\\db^6\\db^7\end{matrix}\right) = -\left(
\begin{matrix}
0 &-\bn&0&\sqrt 2 \bv&0&-\by_1&-\by_2\\
0&0&\bn&0&\by_1&0&\bv\\
0&0&0&\sqrt 2 \bn&\by_2&-\bv&0\\
0&0&0&0&\sqrt 2 \bv&0&\sqrt 2 \bn\\
0&0&0&0&0&0&0\\
0&0&0&0&\bn&0&0\\
0&0&0&0&0&-\bn&0
\end{matrix}\right) \wedge \left(\begin{matrix} b^1\\b^2\\b^3\\b^4\\b^5\\b^6\\b^7\end{matrix}\right).
\end{gather}

\begin{Proposition} The holonomy of $\big(M^{4,3},g\big)$ is contained in the Lie algebra $\fh$ of Type~I~$4(b)$ with $j=0$ if and only if there are local coordinates $x_1,\dots,x_7$ such that $g=2\big(b^1\cdot b^5+b^2\cdot b^6+b^3\cdot b^7\big)- \big(b^4\big)^2$ for
\begin{gather*}
b^1=dx_1 + r_5(x_2,x_4,x_5,x_6, x_7)dx_5+r_6(x_3,x_5,x_6,x_7)dx_6+r_7(x_4,x_5,x_6,x_7)dx_7, \\
b^3= dx_3+q(x_5,x_6,x_7) dx_6, \qquad
b^7= dx_7 +p(x_5,x_6)dx_5,\qquad
b^j= dx_j, \qquad j=2,4,5,6,
\end{gather*}
where
$p=p(x_5,x_6)$ is arbitrary and the functions $q$, $r_5$, $r_6$, $r_7$ are of the form
\begin{gather*}
q=-p_{x_6}x_7+\bar q, \\
r_5=p_{x_6}x_2+ \tfrac{\sqrt2}2 \big( (\bar r_6)_{x_7}-(\bar r_7)_{x_6}-3p_{x_6}p-p_{x_6x_5}x_7+\bar q_{x_5}\big)x_4+p_{x_6 x_6}x_4^2+\bar r_5,\\
r_6=-p_{x_6}x_3+\bar r_6, \qquad
r_7=-2\sqrt2 p_{x_6}x_4+\bar r_7.
\end{gather*}
where
$\bar q=\bar q(x_5,x_6)$, $\bar r_j =\bar r_j (x_5,x_6,x_7)$, $j=5,6,7$, are arbitrary.
\end{Proposition}

\begin{proof} We proceed as in the previous cases integrating the structure equations (\ref{EI4(b)}). We start by choosing $b^5=dx_5$, $b^6=dx_6$, which implies $\bn\in I(dx_5)$. Moreover, we choose $b^7= dx_7 +p(x_5,x_6)dx_5$, $b^4=dx_4$, thus $\bv\in I(dx_5,dx_7)$. Furthermore, we may choose $b^3= dx_3+qdx_6$. Then
\begin{gather*}dq\wedge dx_6=-\sqrt 2 \bn\wedge b^4 -\by_2\wedge b^5+\bv\wedge b^6 \in I\big(dx_{(4,5)},\by_2\wedge dx_5\big)+I_1, \end{gather*}
which yields $q=q(x_5,x_6,x_7)$ and $\by_2\in I(dx_4)+I_0$. Next we choose $b^2=dx_2$, which gives $\by_1\in I(dx_3)+I_0$. Finally, $b^1= dx_1 + r_5dx_5+r_6dx_6+r_7dx_7$. Then
\begin{gather*}
dr_5\wedge dx_5+dr_6\wedge dx_6+dr_7\wedge dx_7 = \bn\wedge b^2 -\sqrt 2 \bv\wedge b^4+\by_1\wedge b^6+\by_2\wedge b^7\\
\hphantom{dr_5\wedge dx_5+dr_6\wedge dx_6+dr_7\wedge dx_7}{} \in I\big(dx_{(2,5)}, dx_{(3,6)}, dx_{(4,5)}, dx_{(4,7)}\big) +I_1,
\end{gather*}
hence $r_5=r_5(x_2,x_4,x_5,x_6, x_7)$, $r_6=r_6(x_3,x_5,x_6,x_7)$, $r_7=r_7(x_4,x_5,x_6,x_7)$.

Now let the metric $g$ be defined by (\ref{Eip}) with respect to the local coordinates that we considered above. Then~$\nabla b_1=0$ and
\begin{gather*}
\nabla b_2= (r_5)_{x_2} b^5\otimes b_1,\\
\nabla b_3= \tfrac12 \big(p_{x_6}+(r_6)_{x_3}\big)b^6\otimes b_1 - \tfrac12 \big(p_{x_6}-(r_6)_{x_3}\big)b^5\otimes b_2,\\
\nabla b_4= \big( \big((r_5)_{x_4}-(r_7)_{x_4}p\big)b^5+ \tfrac12 (r_7)_{x_4}b^7\big)\otimes b_1 + \tfrac12 (r_7)_{x_4}b^5\otimes b_3,\\
b^2(\nabla b_7)= \tfrac12 \big( (r_6)_{x_7}-(r_7)_{x_6}-q_{x_7}p+q_{x_5}\big)b^5+q_{x_7}b^7.
\end{gather*}
Hence the holonomy of $g$ is contained in $\fh$ if and only if
\begin{gather*}
(r_5)_{x_2}=p_{x_6},\qquad (r_6)_{x_3}=-p_{x_6},\qquad 2\sqrt2 q_{x_7}=(r_7)_{x_4}=-2\sqrt2 p_{x_6},\\
(r_5)_{x_4}-(r_7)_{x_4}p= \tfrac{\sqrt2}2 \big((r_6)_{x_7}-(r_7)_{x_6}-q_{x_7}p+q_{x_5}\big),
\end{gather*}
which is equivalent to the assertion.
\end{proof}
\begin{Remark} Proceeding similarly as in the previous cases we see that a general $\fh$-adapted coframe depends on 2~functions in 2~variables and 3~functions in 3~variables.
\end{Remark}
\begin{Example} Starting with $p(x_5,x_6)=x_5 x_6+x_6^2$, $\bar q=\bar r_5=0$, $\bar r_6(x_5,x_6,x_7)=x_5^2$ and $\bar r_7(x_5,x_6,x_7)=x_6^2$ we obtain
\begin{gather*}
q(x_5,x_6,x_7) = -(x_5+2x_6) x_7 ,\\
r_5(x_2,x_4,\dots, x_7) = (x_5+2x_6)x_2- \tfrac1{\sqrt2} \big(2x_6+3x_5^2x_6+9x_5x_6^2+6x_6^3+x_7\big) x_4 +2 x_4^2,\\
r_6(x_3,x_5,x_6,x_7)= -(x_5+2x_6) x_3+x_5^2,\qquad
r_7(x_4,\dots, x_7)= -2\sqrt 2(x_5+2x_6)x_4+ x_6^2 .
\end{gather*}
Then one computes
\begin{gather*}R_{56}=h\big(2N,*,(0,2)^\top,*\big),\qquad R_{57}=- \tfrac12 h(0,1,0,0),\qquad R_{36}= h\big(0,0,0,(2,0)^\top\big).\end{gather*}
Since these elements generate $\fh$ as a Lie algebra, the holonomy algebra coincides with~$\fh$.
\end{Example}

\subsection[Type I 4(b), $j=1$]{Type I 4(b), $\boldsymbol{j=1}$} \label{S3.9}
We consider
\begin{gather*}\fh=\RR\cdot h\big(N,0,(0,1)^\top,0\big) \ltimes\fm(1,1,2).\end{gather*}

The structure equations are
\begin{gather}\label{EI4(b)j=1}\left(\begin{matrix} db^1\\db^2\\db^3\\db^4\\db^5\\db^6\\db^7\end{matrix}\right) = -\left(
\begin{matrix}
0 &-\bn&\bu_1&\sqrt 2 \bv&0&-\by_1&-\by_2\\
0&0&\bn&\sqrt2 \bu_1&\by_1&0&\bv\\
0&0&0&\sqrt 2 \bn&\by_2&-\bv&0\\
0&0&0&0&\sqrt 2 \bv&\sqrt2 \bu_1&\sqrt 2 \bn\\
0&0&0&0&0&0&0\\
0&0&0&0&\bn&0&0\\
0&0&0&0&-\bu_1&-\bn&0
\end{matrix}\right) \wedge \left(\begin{matrix} b^1\\b^2\\b^3\\b^4\\b^5\\b^6\\b^7\end{matrix}\right).
\end{gather}
\begin{Proposition} The holonomy of $\big(M^{4,3},g\big)$ is contained in the Lie algebra $\fh$ of Type~I~$4(b)$ with $j=1$
if and only if there are local coordinates $x_1,\dots,x_7$ such that $g=2\big(b^1\cdot b^5+b^2\cdot b^6+b^3\cdot b^7\big)- \big(b^4\big)^2$ for
\begin{gather*}
b^1= dx_1 + r_5(x_2,\dots, x_7)dx_5+r_6(x_3,\dots,x_7)dx_6+r_7(x_4,\dots,x_7)dx_7, \nonumber \\
b^2= dx_2 +q_2(x_4,x_5,x_6,x_7) dx_6, \qquad
b^3= dx_3+q_3(x_5,x_6,x_7) dx_6 ,\\ 
b^j= dx_j, \qquad j=4,5,6,7,
\end{gather*}
where $q_3$ is arbitrary and the functions $q_2$ and $r_5$, $r_6$, $r_7$ are of the form
\begin{gather*}
q_2 = \sqrt2 (q_3)_{x_7}x_4 +\bar q_2,\\
r_5 = -(q_3)_{x_7}x_2+\sqrt2 (q_3)_{x_7 x_7}x_3 x_4+(\bar q_2)_{x_7}x_3+ \tfrac{\sqrt2}3(q_3)_{x_7 x_7 x_7}x_4^3+\big((\bar q_2)_{x_7 x_7}-(q_3)_{x_6 x_7}\big)x_4^2 \\
 \hphantom{r_5 =}{} + \tfrac1 {\sqrt2}\big((\bar r_6)_{x_7}-(\bar r_7)_{x_6}+(q_3)_{x_5}\big)x_4+\bar r_5, \\
r_6 = 2(q_3)_{x_7}x_3+2(q_3)_{x_7 x_7}x_4^2+2\sqrt2(\bar q_2)_{x_7} x_4+\bar r_6, \qquad
r_7 = 2\sqrt2 (q_3)_{x_7}x_4+\bar r_7,
\end{gather*}
where $\bar q_2=\bar q_2(x_5,x_6,x_7)$, $\bar r_j =\bar r_j (x_5,x_6,x_7)$, $j=5,6,7$, are arbitrary.
\end{Proposition}

\begin{proof} We integrate (\ref{EI4(b)j=1}). We may choose $b^j=dx_j$, $j=5,6,7$, which implies $\bn\in I(dx_5)$ and $\bu_1\in I(dx_5, dx_6)$. Now (\ref{EI4(b)j=1}) gives $db^4\in I(dx_5)$. Hence we may choose $b^4=dx_4$. Then $\bv\in I_0$. Again using (\ref{EI4(b)j=1}) we see that we may choose $b^k=dx_k+q_k dx_6$, $k=2,3$. Because of
\begin{gather*} dq_3\wedge dx_6=-\sqrt 2 \bn\wedge dx_4 -\by_2\wedge dx_5 +\bv\wedge dx_6\in I\big(dx_{(4,5)}, \by_2\wedge dx_5\big) +I_1\end{gather*}
we obtain $q_3=q_3(x_5,x_6,x_7)$ and $\by_2\in I(dx_4)+I_0$. Furthermore, we have
\begin{gather*}
 dq_2\wedge dx_6 =-\bn\wedge b^3-\sqrt 2 \bu_1\wedge dx_4 -\by_1\wedge dx_5 -\bv\wedge dx_7\\
\hphantom{dq_2\wedge dx_6}{} \in I\big(dx_{(3,5)},dx_{(4,5)},dx_{(4,6)}, \by_1\wedge dx_5\big) +I_1.
 \end{gather*}
This implies $q_2=q_2(x_4,x_5,x_6,x_7)$ and $\by_1\in I(dx_3,dx_4)+I_0$. Finally, (\ref{EI4(b)j=1}) shows $db^1\in I_0$, thus $b^1= dx_1 + r_5dx_5+r_6dx_6+r_7dx_7$. Then
\begin{gather*}
dr_5\wedge dx_5+dr_6\wedge dx_6+dr_7\wedge dx_7 = \bn\wedge b^2 -\bu_1\wedge b^3 - \sqrt 2 \bv\wedge dx_4+\by_1\wedge dx_6+\by_2\wedge dx_7\\
\hphantom{dr_5\wedge dx_5+dr_6\wedge dx_6+dr_7\wedge dx_7}{} \in I(dx_{(2,5)}, dx_{(3,5)}, dx_{(3,6)}, dx_{(4,5)}, dx_{(4,6)}, dx_{(4,7)}) +I_1,
\end{gather*}
hence $r_5=r_5(x_2,\dots, x_7)$, $r_6=r_6(x_3,\dots,x_7)$, $r_7=r_7(x_4,\dots,x_7)$.

Let the metric $g$ be defined by (\ref{Eip}) with respect to the local coordinates that we considered above. Then $\nabla b_1=0$ and
\begin{gather*}
\nabla b_2 = (r_5)_{x_2} b^5\otimes b_1,\qquad
\nabla b_3 = \big((r_5)_{x_3}b^5+\tfrac12 (r_6)_{x_3}b^6\big)\otimes b_1+\tfrac12 (r_6)_{x_3}b^5\otimes b_2,\\
\nabla b_4 = \big( (r_5)_{x_4}b^5+ \tfrac12 (r_6)_{x_4}b^6+\tfrac12 (r_7)_{x_4}b^7\big)\otimes b_1 + \big(\tfrac12 (r_6)_{x_4}b^5+(q_2)_{x_4} b^6\big)\otimes b_2\\
\hphantom{\nabla b_4 =}{} + \tfrac12 (r_7)_{x_4} b^5\otimes b_3,\\
b^2(\nabla b_7) = \tfrac12 \big( (r_6)_{x_7}-(r_7)_{x_6}+(q_3)_{x_5} \big)b^5+ (q_2)_{x_7}b^6 +(q_3)_{x_7}b^7.
\end{gather*}
Hence the holonomy of $g$ is contained in $\fh$ if and only if
\begin{gather*}
-2\sqrt2(r_5)_{x_2} =\sqrt2(r_6)_{x_3}=(r_7)_{x_4}, \qquad
{2\sqrt2}(r_5)_{x_3} =(r_6)_{x_4},\qquad (r_6)_{x_3}= {\sqrt2}(q_2)_{x_4}, \\
\sqrt2 (r_5)_{x_4} = (r_6)_{x_7}-(r_7)_{x_6}+(q_3)_{x_5},\qquad (r_6)_{x_4}=2\sqrt2 (q_2)_{x_7},\qquad (r_7)_{x_4}=2\sqrt2 (q_3)_{x_7},
\end{gather*}
which is equivalent to the assertion.
\end{proof}
\begin{Remark} Proceeding similarly as in the previous cases we see that in this case a general $\fh$-adapted coframe depends on 5~functions in 3~variables.
\end{Remark}

\begin{Example} Starting with $q_3(x_5,x_6,x_7)=(x_5+ x_6+x_7)x_7$, $\bar q_2(x_5,x_6,x_7)=2x_6x_7$, $\bar r_5=\bar r_6=\bar r_7=0$, we obtain
\begin{gather*}
q_2(x_4,\dots,x_7) = \sqrt2(x_5+x_6+2 x_7)x_4+2x_6 x_7 ,\\
r_5(x_2,\dots, x_7) = -(x_5+x_6+2 x_7)x_2+2 \sqrt2 x_3x_4+2 x_3x_6-x_4^2+ \tfrac1{\sqrt2} x_4x_7,\\
r_6(x_3,\dots,x_7) = 2(x_5+x_6+2x_7) x_3+4x_4^2+4\sqrt2 x_4 x_6,\\
r_7(x_4,\dots, x_7) = 2\sqrt 2(x_5+x_6+2x_7)x_4 .
\end{gather*}
To show that the holonomy is equal to $\fh$ we again compute parts of the curvature tensor:
\begin{gather*}
 R_{56}=-h\big(N,*,(1,1)^\top,*\big),\qquad R_{35}=h(0,2,0,*),\\
 R_{67}= h\big(0,*,(-2,0)^\top,*\big),\qquad R_{25}= h\big(0,0,0,(1,2)^\top\big),
 \end{gather*}
which generate $\fh$ as a Lie algebra.
\end{Example}

\subsection[Holonomy algebras containing $\fm$]{Holonomy algebras containing $\boldsymbol{\fm}$}
In the previous sections we concentrated on holonomy groups that are either maximal, i.e., isomorphic to $\fgl(2,\RR)\ltimes\fm$ or that are small in the sense that they do not contain the whole Lie algebra~$\fm$. For the sake of completeness we consider now also the remaining Berger algebras~$\fh$ of Type~I satisfying $\fm\subset \fh$. We will not give normal forms for these metrics here, but we will give an example of a~metric for each of these Lie algebras. This will complete the proof of Theorem~\ref{T}.

In all cases we will proceed as follows. Given a Berger algebra $\fh=\fa\ltimes\fm$, we choose functions $r_5$, $r_6$, $r_7$, $q_2$, $q_3$, $q_4$, $s_2$, $s_3$ and~$f$ that satisfy the differential equations (\ref{Eb1})--(\ref{Eb10}) and in addition the differential equations for $b^i(\nabla b_j)$ that are equivalent to the condition that the projection of the connection form to $\fgl(2,\RR)$ is contained in~$\fa$. This will ensure that the holonomy of the corresponding metric is contained in~$\fh$. Finally, in each case, one has to check that the holonomy is equal to $\fh$. This is done in the same way as for the various examples in Sections~\ref{S3.2}--\ref{S3.9}. If $\fa$ is equal to $\fco(2)$, $\fb_2$, $\fd$, $\RR\cdot C_a$, $\RR\cdot S$ or~$\RR\cdot \diag(1,\mu)$ for $\mu\not=0$, the calculations are easy since~$\fh$ is already generated (as a Lie algebra) by $\{R_{ij}\,|\, i<j\}$. If $\fa$ equals $\fsl(2,\RR)$, $\hat b_2$, $\RR\cdot\diag(1,0)$ or~$\fs_\lambda$, then we need also~$\nabla R$ to generate~$\fh$.

According to the described approach we now give functions $f$, $s_2$, $s_3$, $q_2$, $q_3$, $q_4$, $r_5$, $r_6$, $r_7$ for every $\fa\subset\fgl(2,\RR)$ for which $\fa\ltimes \fm$ is on the list in Theorem~\ref{T1}. We start with $\fa=\fsl(2,\RR)$:
\begin{gather*}
f(x_5,x_6,x_7) = e^{x_5},\qquad
s_2(x_2,\dots,x_7) = 0,\qquad
s_3(x_2,\dots,x_7) = x_2+x_4,\\
q_2(x_2,\dots,x_7) = \tfrac14 \sqrt2 x_4-1,\qquad
q_3(x_2,\dots,x_7) = -\sqrt2 e^{x_5}+x_6^2,\qquad
q_4(x_5,x_6,x_7) = 1,\\
r_5(x_1,\dots,x_7) = -\tfrac14 e^{x_5} \sqrt2 x_2x_4+x_6x_7,\qquad
r_6(x_1,\dots,x_7) = \tfrac12 x_3, \\
r_7(x_1,\dots,x_7) = \tfrac12 x_4^2- \sqrt2 x_2 .
\end{gather*}

We continue with $\fa\in \{\fb_2, \fd\}$. Starting with the functions
\begin{gather*}
f(x_5,x_6,x_7) = e^{x_6} ,\qquad
s_2(x_2,\dots, x_7) = x_4 x_7-x_4,\qquad
 s_3 (x_2, \dots, x_7) = 0,\\
q_3(x_2, \dots, x_7) = x_4 x_7,\qquad
q_4 (x_5,x_6,x_7) = x_7,\\
r_5(x_1,\dots, x_7) = e^{x_6} \big( x_4^2 x_7^2 - x_3 x_4 x_7 + x_3 x_7^2 - \tfrac23 x_4^3 + 2 x_4^2 x_7 - \tfrac { \sqrt{2}}2 x_1 \\
\hphantom{r_5(x_1,\dots, x_7) =}{} - x_2 x_4 + x_3 x_4 - x_3 x_7 - x_4^2\big),\\
r_6(x_2,\dots, x_7) = \sqrt{2} (-x_4^2 x_7 + 2 x_4 x_7^2 - x_2 x_7 + x_4^2 - 2 x_4 x_7) + x_1 ,\\
r_7(x_4,\dots, x_7) = \sqrt 2 (x_3 x_7 + x_4^2 - x_3),
\end{gather*}
and choosing respectively $q_2 (x_2, \dots, x_7) = x_2 + e^{x_3}$ or $q_2 (x_2, \dots, x_7) = x_2 $ we get holonomy equal to $\fb_2\ltimes \fm $ or $\fd\ltimes \fm$, respectively.

Recall that we already gave an example of a metric with holonomy $\RR\cdot N\ltimes \fm$ in \cite{FK}. It remains to consider those Berger algebras $\fh=\fa\ltimes \fm$ for which $\fa$ is equal to $\RR\cdot S$, $\hat \fb_2$, $\fco(2)$, $\RR\cdot C_a$, $\RR\cdot\diag(1,\mu)$, $\mu\in[-1,1]$ or to $\fs_\lambda$, $\lambda\in\RR$. We may assume $\lambda\not=1/2$ since we considered the case $\lambda=1/2$ already in \cite{FK}.

For all these $\fa$ we choose $f(x_5,x_6,x_7) = e^{x_5}$. Furthermore, \begin{gather*}s_2(x_2,\dots, x_7)=\alpha x_4x_6,\qquad q_3(x_2,\dots, x_7)=\beta x_4x_6,\end{gather*} where
\begin{gather*} (\alpha,\beta) = \begin{cases}
\big({-}\tfrac12, \tfrac12\big), & \text{if } \fa\in \{\RR\cdot S,\, \hat \fb_2,\, \fco(2) \},\\
(-a,a),& \text{if } \fa=\RR\cdot C_a,\\
(1,1),& \text{if } \fa=\RR\cdot\diag(1,-1),\\
\big({-}\tfrac{\mu}{1+\mu}, \tfrac1{1+\mu}\big), & \text{if } \fa=\RR\cdot\diag(1,\mu),\ \mu\in(-1,1],\\
 \big({-}\tfrac{\lambda-1}{2\lambda -1}, \tfrac \lambda{2\lambda-1}\big), & \text{if } \fa=\fs_\lambda.\\
\end{cases}\end{gather*}
Moreover,
\begin{gather*}
 s_3 (x_2, \dots, x_7) = \begin{cases}
x_6x_7,& \text{if } \fa \in\{ \hat \fb_2,\, \fs_\lambda ,\, \RR\cdot S,\, \RR\cdot\diag(1,\mu) \,|\,\mu\in [-1,1] \}, \\
x_6x_7+x_4x_7,& \text{if } \fa =\fco(2) , \\
x_6x_7+x_4x_6,& \text{if } \fa =\RR\cdot C_a ,
\end{cases} \\
q_2 (x_2, \dots, x_7) = \begin{cases}
-\tfrac12 x_4x_6,& \text{if } \fa=\RR\cdot S, \\
x_3x_5,& \text{if } \fa \in\{\hat\fb_2,\fs_\lambda\} , \\
0,& \text{if } \fa =\RR\cdot \diag(1,\mu) , \ \mu\in[-1,1], \\
x_4x_7,& \text{if } \fa =\fco(2),\\
x_4x_6,& \text{if } \fa =\RR\cdot C_a,
\end{cases}\\
q_4(x_5,x_6, x_7)=\begin{cases}x_6x_7,& \text{if } \fa=\RR\cdot S,\\
0,& \text{if } \fa=\RR\cdot \diag(1,-1),\\
2ax_6(x_5+x_7) ,& \text{if } \fa=\RR\cdot C_a,\\
x_6(x_5+x_7),& \text{else.}
\end{cases}
\end{gather*}
In order to define $r_5$, $r_6$, $r_7$, we put
\begin{gather*}
\rho_1(x_1,\dots,x_7) := -\tfrac1{\sqrt2}\big(x_1e^{x_5}+x_3\big)x_6+x_2x_7e^{x_5} ,\\
\rho_2(x_1,\dots,x_7) := -\tfrac12\big(\big((x_5+x_7)x_3+\tfrac32 x_4^2\big)x_6^2-x_3x_4\big)e^{x_5}, \\
\rho_3(x_1,\dots,x_7) := \sqrt 2 \big(\tfrac23 x_7^3+x_5x_7^2-x_2\big)x_6-e^{-x_5}x_4x_6 ,\\
\rho_4(x_1,\dots,x_7) := -\sqrt 2x_4(x_5+x_7)x_6^2+\tfrac 1{\sqrt 2}\big(x_2x_6+x_4^2\big).
\end{gather*}
Then, for $\fa=\RR\cdot S$,
\begin{gather*}
r_5(x_1,\dots, x_7) = \rho_2(x_1,\dots,x_7)+ \big({-}\tfrac 1{\sqrt2} x_1x_6+x_2x_7-2x_4x_6x_7^2+\tfrac12(x_3x_5+x_4x_7)x_6^2\big)e^{x_5},\\
r_6(x_1,\dots, x_7) = -\tfrac 1{\sqrt2}\big( 2x_4x_6^2x_7 +(x_2+x_3)x_6-x_4^2\big) , \\
r_7(x_1,\dots, x_7) = -\tfrac1{\sqrt2}(x_3x_6+4x_4x_7).
\end{gather*}
For $\fa=\hat \fb_2$,
\begin{gather*}
r_5(x_1,\dots, x_7) = (\rho_1+\rho_2)(x_1,\dots,x_7)-\big(x_3x_7+x_4^2\big)x_5x_6 e^{x_5}, \\
r_6(x_1,\dots, x_7) = (\rho_3+\rho_4)(x_1,\dots,x_7)-2\sqrt2 x_4x_5x_6x_7 , \\
r_7(x_1,\dots, x_7) = -\tfrac1{\sqrt2}(x_3x_6+4x_4x_7).
\end{gather*}
If $\fa=\RR\cdot\diag(1,\mu)$, $\mu\not=-1$, then
\begin{gather*}
r_5(x_1,\dots, x_7) = \big(\rho_1+\tfrac{2\mu}{\mu+1}\rho_2\big)(x_1,\dots, x_7)+\tfrac{\mu(\mu-1)}{2(1+\mu)^2}x_4^2x_6^2e^{x_5}, \\
r_6(x_1,\dots, x_7) = \big(\rho_3+\tfrac{2\mu}{\mu+1}\rho_4\big)(x_1,\dots, x_7), \qquad
r_7(x_1,\dots, x_7) = -\sqrt2\big(\tfrac{\mu}{1+\mu}x_3x_6 +2x_4x_7\big).
\end{gather*}
For $\fa=\RR\cdot\diag(1,-1)$,
\begin{gather*}
r_5(x_1,\dots, x_7) = \big(x_4^2x_6^2+x_2x_7-x_3x_4\big)e^{x5}, \qquad
r_6(x_1,\dots, x_7) = -\sqrt2\big(x_2x_6+x_4^2\big), \\
r_7(x_1,\dots, x_7) = \sqrt2(x_3x_6-2x_4x_7).
\end{gather*}
If $\fa=\fs_\lambda$, $\lambda\not=1/2$, then
\begin{gather*}
r_5(x_1,\dots, x_7) = \big(\rho_1+\tfrac{2(\lambda-1)}{2\lambda-1}\rho_2\big)(x_1,\dots,x_7) -\tfrac{\lambda-1}{2(2\lambda-1)^2}x_4^2x_6^2e^{x_5}-\big(x_3x_7+x_4^2\big)x_5x_6e^{x_5}, \\
r_6(x_1,\dots, x_7) = \big(\rho_3+\tfrac{2(\lambda-1)}{2\lambda-1}\rho_4\big)(x_1,\dots,x_7)-2\sqrt2 x_4x_5x_6x_7,\\
r_7(x_1,\dots, x_7) = -\sqrt2 \big(\tfrac{\lambda-1}{2\lambda -1}x_3x_6+2 x_4 x_7\big).
\end{gather*}
For $\fa=\fco(2)$,
\begin{gather*}
r_5(x_1,\dots, x_7) = \big(x_3x_4-\big(x_2x_6+x_4^2\big)(x_5+x_7)x_7-x_4^2x_7^2\big)e^{x_5} +(\rho_1+\rho_2)(x_1,\dots,x_7),\\
r_6(x_1,\dots, x_7) = \sqrt2\big(x_3x_7+x_4^2-x_6^2x_7^2\big(x_5^2+\tfrac43 x_5x_7+\tfrac12x_7^2\big)+\tfrac13x_6x_7^3\big) \\
\hphantom{r_6(x_1,\dots, x_7) =}{} +(\rho_3+\rho_4)(x_1,\dots,x_7) ,\\
r_7(x_1,\dots, x_7) = \sqrt2\big(2x_4x_6x_7(x_5+x_7)-x_2x_7-2x_4x_7-\tfrac12 x_3x_6\big),
\end{gather*}
and for $\fa=\RR\cdot C_a$, we put
\begin{gather*}
r_5(x_1,\dots, x_7) = \big(x_2x_4-2ax_6\big(x_2x_6+3x_4^2\big)(x_5+x_7)+(1 - 2a)(ax_3x_4+x_2x_7)-x_4^2x_6^2\big)e^{x_5} \\
 \hphantom{r_5(x_1,\dots, x_7) =}{} +2a(\rho_1+2a\rho_2)(x_1,\dots, x_7),\\
r_6(x_1,\dots, x_7) = \sqrt2\big(x_6x_3-8a^2x_5x_6^3x_7(x_5+x_7)+\big(a-2a^2\big)\big(x_2x_6+x_4^2\big)\big) \\
 \hphantom{r_6(x_1,\dots, x_7) =}{} -\tfrac{\sqrt2}6x_6^2x_7^2\big(16a^2x_6x_7-3\big) +2a(\rho_3+2a\rho_4)(x_1,\dots, x_7),\\
r_7(x_1,\dots, x_7) = -\sqrt 2\big(\big(x_4-4a(x_5+x_7)x_6^2+2x_7\big)x_4+x_6(ax_3+x_2)\big).
\end{gather*}

\subsection*{Acknowledgements} We thank the referees for their valuable and constructive comments.

\pdfbookmark[1]{References}{ref}
\LastPageEnding

\end{document}